\documentclass[12pt]{amsart}
\usepackage{amssymb,tikz,hyperref}
\theoremstyle{plain}
\newtheorem{theorem}{Theorem}[section]
\newtheorem{claim}{Claim}[theorem]
\newtheorem{lemma}[theorem]{Lemma}
\newtheorem{prop}[theorem]{Proposition}
\theoremstyle{remark}
\newtheorem{rmk}[theorem]{Remark}
\theoremstyle{definition}
\newtheorem{defn}[theorem]{Definition}
\newtheorem{example}[theorem]{Example}

\numberwithin{equation}{section}

 \DeclareMathOperator{\id}{id} 
  \DeclareMathOperator{\Obj}{Obj}
\DeclareMathOperator{\Mor}{Mor}  
\DeclareMathOperator{\dom}{dom}  \DeclareMathOperator{\lsp}{span}
\DeclareMathOperator{\MCE}{MCE}  
 \DeclareMathOperator{\clsp}{{\overline{span}}}

\DeclareMathOperator{\qc}{{\square_{Claim}}}

\newcommand{\Lmin}{\Lambda^{\min}}
\newcommand{\ds}{\displaystyle}
\newcommand{\eps}{\varepsilon}
\newcommand{\wt}[1]{\widetilde{#1}}
\newcommand{\pc}{\renewcommand{\qedsymbol}{$\qc$}}

\newcommand{\field}[1]{\mathbb{#1}}
\newcommand{\CC}{\field{C}}

\newcommand{\NN}{\field{N}}

\newcommand{\TT}{\field{T}}

\newcommand{\Gg}{{\mathcal G}}

\newcommand{\Zz}{{\mathcal Z}}

\newcommand{\FE}{{\mathcal{FE}}}

\title{The path space of a higher-rank graph}

\author[S.B.G. Webster]{Samuel B.G. Webster}
\address{Samuel Webster\\
School of Mathematics and Applied Statistics\\
University of Wollongong\\
NSW 2522\\
AUSTRALIA}
\email{sbgwebster@gmail.com}

\subjclass[2010]{Primary 46L05}
\keywords{Graph algebra, k-graph, higher-rank graph}
\thanks{This research was supported by the ARC Discovery Project DP0984360.}
\date{29th March, 2011.}

\begin{document}
\begin{abstract}
We construct a locally compact Hausdorff topology on the path space of a finitely aligned $k$-graph $\Lambda$. We identify the boundary-path space $\partial\Lambda$ as the spectrum of a commutative $C^*$-subalgebra $D_\Lambda$ of $C^*(\Lambda)$. Then, using a construction similar to that of Farthing, we construct a finitely aligned $k$-graph $\wt\Lambda$ with no sources in which $\Lambda$ is embedded, and show that $\partial\Lambda$ is homeomorphic to a subset of $\partial\wt\Lambda$ . We show that when $\Lambda$ is row-finite, we can identify $C^*(\Lambda)$ with a full corner of $C^*(\wt\Lambda)$, and deduce that $D_\Lambda$ is isomorphic to a corner of $D_{\wt\Lambda}$. Lastly, we show that this isomorphism implements the homeomorphism between the boundary-path spaces.
\end{abstract}

\maketitle
%\tableofcontents

\section{Introduction}
Cuntz and Krieger's work~\cite{CK1980} on $C^*$-algebras associated to $(0,1)$-matrices and the subsequent interpretation of Cuntz and Krieger's results by Enomoto and Watatani~\cite{EW1980} were the foundation of the field we now call graph algebras. Directed graphs and their higher-rank analogues provide an intuitive framework for the analysis of this broad class of $C^*$-algebras; there is an explicit relationship between the dynamics of a graph and various properties of its associated $C^*$-algebra. Kumjian and Pask in \cite{KP2000} introduced \emph{higher-rank graphs} (or \emph{$k$-graphs}) as analogues of directed graphs in order to study Robertson and Steger's higher-rank Cuntz-Krieger algebras \cite{RS1999a} using the techniques previously developed for directed graphs. Higher-rank graph $C^*$-algebras have received a great deal of attention in recent years, not least because they extend the already rich and tractable class of graph $C^*$-algebras to include all tensor products of graph $C^*$-algebras (and thus many Kirchberg algebras whose $K_1$ contains torsion elements \cite{KP2000}), as well as (up to Morita equivalence) the irrational rotation algebras and many other examples of simple A$\TT$-algebras with real rank zero \cite{PRRS2006}.

Although the definition of a $k$-graph (Definition~\ref{hrgdef}) isn't quite as straightforward as that of a directed graph, $k$-graphs are a natural generalisation of directed graphs: Kumjian and Pask show in  \cite[Example 1.3]{KP2000} that $1$-graphs are precisely the path-categories of  directed graphs. Like directed graph $C^*$-algebras, higher-rank graph $C^*$-algebras were first studied using groupoid techniques. Kumjian and Pask defined the $k$-graph $C^*$-algebra $C^*(\Lambda)$ to be the universal $C^*$-algebra for a set of Cuntz-Krieger relations among partial isometries associated to paths of the $k$-graph $\Lambda$. Using direct analysis, they proved a version of the gauge-invariant uniqueness theorem for $k$-graph algebras. They then constructed a groupoid $\Gg_\Lambda$ from each $k$-graph $\Lambda$, and used the gauge invariant uniqueness theorem to prove that the groupoid $C^*$-algebra $C^*(\Gg_\Lambda)$ is isomorphic to $C^*(\Lambda)$. This allowed them to make use of Renault's theory of groupoid $C^*$-algebras to analyse $C^*(\Lambda)$.  The unit space $\Gg_\Lambda^{(0)}$ of $\Gg_\Lambda$, which must be locally compact and Hausdorff, is a collection of paths in the graph: for a row-finite graph with no sources, $\Gg_\Lambda^{(0)}$ is the collection of infinite paths in $\Lambda$ (the definition of an infinite path in a $k$-graph is not straightforward, see Remark~\ref{paths are morphisms}). For more complicated graphs, the infinite paths are replaced with \emph{boundary paths} (Definition~\ref{defn boundary paths}).

In \cite{RSY2003}, Raeburn, Sims and Yeend developed a ``bare-hands'' analysis of $k$-graph $C^*$-algebras. They found a slightly weaker alternative to the no-sources hypothesis from Kumjian and Pask's theorems called \emph{local convexity} (Definition~\ref{defn: restrictions}). The same authors later introduced \emph{finitely aligned} $k$-graphs in \cite{RSY2004}, and gave a direct analysis of their $C^*$-algebras. This remains the most general class of $k$-graphs to which a $C^*$-algebra has been associated and studied in detail.

Many results for row-finite directed graphs with no sources can be extended to arbitrary graphs via a process called \emph{desingularisation}. Given an arbitrary directed graph $E$, Drinen and Tomforde show in \cite{DT2005} how to construct a row-finite directed graph $F$ with no sources by adding vertices and edges to $E$ in such a way that the $C^*$-algebra associated to $F$ contains the $C^*$-algebra associated to $E$ as a full corner. The modified graph $F$ is now called a \emph{Drinen-Tomforde desingularisation of $E$}. Although no analogue of a Drinen-Tomforde desingularisation is currently available for higher-rank graphs, Farthing provided a construction in \cite{Farthing2008} analogous to that in \cite{BPRS2000} for removing the sources in a locally convex, row-finite higher-rank graph. The statements of the results of \cite{Farthing2008} do not contain the local convexity hypothesis, but Farthing alerted us to an issue in the proof of \cite[Theorem 2.28]{Farthing2008} (see Remark~\ref{cindys error}), which arises when the graph is not locally convex.

The goal of this paper is to explore the path spaces of higher-rank graphs and investigate how these path spaces interact with desingularisation procedures such as Farthing's.

In Section~\ref{hrg sec}, we recall the definitions and standard notation for higher-rank graphs. In Section~\ref{hrg top}, following the approach of \cite{PW2005}, we build a topology on the path space of a higher-rank graph, and show that the path space is locally compact and Hausdorff under this topology.

In Section~\ref{hrg removing sources}, given a finitely aligned $k$-graph $\Lambda$, we construct a $k$-graph $\wt\Lambda$ with no sources which contains a subgraph isomorphic to $\Lambda$. Our construction is modelled on Farthing's construction in \cite{Farthing2008}, and the reader is directed to \cite{Farthing2008} for several proofs. The crucial difference is that our construction involves extending elements of the boundary-path space $\partial\Lambda$, whereas Farthing extends paths from a different set $\Lambda^{\leq\infty}$ (see Remark~\ref{boudary path remark}). Interestingly, although $\partial\Lambda$ and $\Lambda^{\leq\infty}$ are potentially different when $\Lambda$ is row-finite and not locally convex (Proposition~\ref{leqsubsetpartial}), our construction and Farthing's yield isomorphic $k$-graphs except in the non-row-finite case (Examples~\ref{desource ex} and Proposition~\ref{me and cindy agree when the hrg is rf and lc}). We follow Robertson and Sims' notational refinement \cite{RS2009} of Farthing's desourcification: we construct a new $k$-graph in which the original $k$-graph is embedded, whereas Farthing's construction adds bits onto the existing $k$-graph. This simplifies many arguments involving $\wt\Lambda$; however, the main reason for modifying Farthing's construction is that $\Lambda^{\leq \infty}$ is not as well behaved topologically as $\partial\Lambda$ (see Remark~\ref{lambda leq infty not closed}) and in particular, no analogue of Theorem~\ref{hrg path spaces homeo} holds for Farthing's construction.

In Section~\ref{hrg path spaces under desource}, we prove that given a row-finite $k$-graph $\Lambda$, there is a natural homeomorphism from the boundary-path space of $\Lambda$ onto the space of infinite paths in $\wt \Lambda$ with range in the embedded copy of $\Lambda$. We provide examples and discussion showing that the topological basis constructed in Section~\ref{hrg top} is the one we want.

In Section~\ref{hrg alg} we recall the definition of the Cuntz-Krieger algebra $C^*(\Lambda)$ of a higher-rank graph $\Lambda$. We show that if $\Lambda$ is a row-finite $k$-graph and $\wt \Lambda$ is the graph with no sources obtained by applying the construction of Section~\ref{hrg removing sources} to $\Lambda$, then the embedding of $\Lambda$ in $\wt\Lambda$ induces an isomorphism $\pi$ of $C^*(\Lambda)$ onto a full corner of $C^*(\wt\Lambda)$.

Section~\ref{hrg diag} contains results about the diagonal $C^*$-subalgebra of a $k$-graph $C^*$-algebra: the $C^*$-algebra generated by range projections associated to paths in the $k$-graph. We identify the boundary-path space of a finitely aligned higher-rank graph with the spectrum of its diagonal $C^*$-algebra. We then show that the isomorphism $\pi$ of Section~\ref{hrg alg} restricts to an isomorphism of diagonals which implements the homeomorphism of Section~\ref{hrg path spaces under desource}.

\subsection*{Acknowledgements}
The work contained in this paper is from the author's PhD thesis, and as such I extend thanks to my PhD supervisors Iain Raeburn and Aidan Sims for their support and willingness to proofread and guide my work.

\section{Preliminaries}\label{hrg sec}

\begin{defn}\label{hrgdef}
Given $k \in \NN$, a $k$-graph is a pair $(\Lambda, d)$ consisting of a
countable category $\Lambda = (\Obj(\Lambda), \Mor(\Lambda), r , s)$ together with a functor $d: \Lambda
\to \NN^k$, called the \emph{degree map}, which satisfies the \emph{factorisation property}: for every $\lambda \in \Mor(\Lambda)$ and $m,n \in \NN^k$ with $d(\lambda) = m + n$, there are unique elements $\mu,\nu \in \Mor(\Lambda)$ such that $\lambda = \mu \nu$, $d(\mu) = m$ and $d(\nu) = n$. Elements $\lambda \in \Mor(\Lambda)$ are called \emph{paths}. We follow the usual abuse of notation and write $\lambda \in \Lambda$ to mean $\lambda \in \Mor(\Lambda)$. For $m \in \NN^k$ we define $\Lambda^m := \{ \lambda \in \Lambda : d(\lambda) = m \}$. For subsets $F \subset \Lambda$ and $V \subset \Obj(\Lambda)$, we write $VF := \{ \lambda \in F: r(\lambda) \in V \}$ and $FV := \{ \lambda \in F : s(\lambda) \in V\}$. If $V = \{v\}$, we drop the braces and write $vF$ and $Fv$. A morphism between two $k$-graphs $(\Lambda_1,d_1)$ and $(\Lambda_2,d_2)$ is a functor $f: \Lambda_1 \to \Lambda_2$ which respects the degree maps. The factorisation property allows us to identify $\Obj(\Lambda)$ with $\Lambda^0$. We refer to elements of $\Lambda^0$ as \emph{vertices}.
\end{defn}

\begin{rmk}\label{rmk 1skel}
To visualise a $k$-graph we draw its \emph{$1$-skeleton}: a directed graph with vertices $\Lambda^0$ and edges $\bigcup_{i=1}^k \Lambda^{e_i}$. To each edge we assign a colour determined by the edge's degree. We tend to use $2$-graphs for examples, and we draw edges of degree $(1,0)$ as solid lines, and edges of degree $(0,1)$ as dashed lines.
\end{rmk}

\begin{example}
For $k \in \NN$ and $m \in (\NN\cup\{\infty\})^k$, we define $k$-graphs $\Omega_{k,m}$ as follows. Set $\Obj(\Omega_{k,m}) = \{p \in \NN^k : p_i \leq m_i \text{ for all }i \leq k\}$,
\[
    \Mor(\Omega_{k,m}) = \{(p,q): p,q \in \Obj(\Omega_{k,m})\text{ and } p_i \leq q_i \text{ for all }i \leq k\},
\]
$r(p,q) = p$, $s(p,q) = q$ and $d(p,q) = q-p$, with composition given by $(p,q)(q,t) = (p,t)$. If $m = (\infty)^k$, we drop $m$ from the subscript and write $\Omega_k$. The $1$-skeleton of $\Omega_{2,2}$ is depicted in Figure~\ref{omega2}.
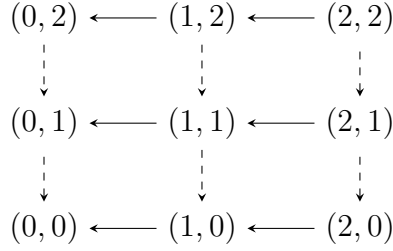
\begin{figure}
\begin{tikzpicture}[>=stealth,scale=0.7]
    \node (u) at (0,2) {$(0,1)$};
    \node (v) at (0,0) {$(0,0)$}
        edge[<-,dashed] node[auto] {} (u);
    \node (w) at (3,0) {$(1,0)$}
        edge[->] node[auto] {} (v);
    \node (x) at (3,2) {$(1,1)$}
        edge[->,dashed] node[auto] {} (w)
        edge[->] node[auto] {} (u);
    \node (a) at (6,4) {$(2,2)$};
    \node (b) at (6,2) {$(2,1)$}
    edge[<-,dashed] node[auto] {} (a)
    edge[->] node[auto] {} (x);
    \node (c) at (3,4) {$(1,2)$}
    edge[<-] node[auto] {} (a)
    edge[->,dashed] node[auto] {} (x);
    \node (d) at (6,0) {$(2,0)$}
    edge[<-,dashed] node[auto] {} (b)
    edge[->] node[auto] {} (w);
    \node (e) at (0,4) {$(0,2)$}
    edge[<-] node[auto] {} (c)
    edge[->,dashed] node[auto] {} (u);
\end{tikzpicture}
\caption{The $2$-graph $\Omega_{2,2}$.}\label{omega2}
\end{figure}
 \end{example}

\begin{rmk}\label{paths are morphisms}
The graphs $\Omega_{k,m}$ provide an intuitive model for paths: every path $\lambda$ of degree $m$ in a $k$-graph $\Lambda$ determines a $k$-graph morphism $x_\lambda: \Omega_{k,m} \to \Lambda$. To see this, let $p,q \in \NN^k$ be such that $p \leq q \leq m$. Define $x_\lambda(p,q) = \lambda''$, where $\lambda = \lambda'\lambda''\lambda'''$; and $d(\lambda') = p$, $d(\lambda'') = q-p$ and $d(\lambda''')=m-q$. In this way, paths in $\Lambda$ are often identified with the graph morphisms $x_\lambda:\Omega_{k,m} \to \Lambda$. We refer to the segment $\lambda''$ of $\lambda$ (as factorized above) as $\lambda(p,q)$, and for $n \leq m$, we refer to the vertex $r(\lambda(n,m)) = s(\lambda(0,n))$ as $\lambda(n)$. By analogy, for $m \in (\NN \cup \{\infty\})^k$ we define $\Lambda^m := \{x:\Omega_{k,m} \to \Lambda : x \text{ is a graph morphism.}\}$. For clarity of notation, if $m = (\infty)^k$ we write $\Lambda^\infty$.
\end{rmk}

Define
\[
    W_\Lambda:= \bigcup_{n \in (\NN \cup \{\infty\})^k} \Lambda^n.
\]
We call $W_\Lambda$ the \emph{path space} of $\Lambda$. We drop the subscript when confusion is unlikely.

For $m,n \in \NN^k$, we denote by $m \wedge n$ the coordinate-wise minimum, and by $m \vee n$ the coordinate-wise maximum. With no parentheses, $\vee$ and $\wedge$ take priority over the group operation: $a-b\wedge c$ means $a - (b \wedge c)$.

Since finite and infinite paths are fundamentally different, that one can compose them isn't immediately obvious.

\begin{lemma}[{\cite[Proposition~3.0.1.1]{WebsterPhD}}]
Let $\Lambda$ be a $k$-graph. Suppose $\lambda \in \Lambda$ and suppose that $x \in W$ satisfies $r(x) = s(\lambda)$. Then that there exists a unique $k$-graph morphism $\lambda x: \Omega_{k,d(\lambda) + d(x)} \to \Lambda$ such that $(\lambda x)(0, d(\lambda)) = \lambda$ and $(\lambda x)(d(\lambda), n+d(\lambda)) = x(0,n)$ for all $n \leq d(x)$.
\end{lemma}

\begin{defn}\label{defn:lmin and mce}
For $\lambda,\mu \in \Lambda$, write
\[
    \Lmin (\lambda,\mu) := \{ (\alpha,\beta) \in \Lambda \times \Lambda : \lambda\alpha = \mu\beta, d(\lambda\alpha) = d(\lambda) \vee d(\mu) \}
\]
for the collection of pairs which give \emph{minimal common extensions} of $\lambda$ and $\mu$, and denote the set of minimal common extensions by
\[
    \MCE(\lambda,\mu):= \{ \lambda\alpha : (\alpha,\beta) \in \Lmin(\lambda,\mu)\}= \{ \mu\beta : (\alpha,\beta) \in \Lmin(\lambda,\mu)\}.
\]
\end{defn}

\begin{defn}\label{defn: restrictions}
A $k$-graph $\Lambda$ is \emph{row-finite} if for each $v \in \Lambda^0$ and $m \in \NN^k$, the set $v\Lambda^m$ is finite; $\Lambda$ has \emph{no sources} if $v\Lambda^m \neq \emptyset$ for all $v \in \Lambda^0$ and $m\in\NN^k$.

We say that $\Lambda$ is \emph{finitely aligned} if $\Lmin (\lambda,\mu)$ is finite (possibly empty) for all $\lambda,\mu \in \Lambda$.

As in \cite[Definition 3.1]{RSY2003}, a $k$-graph $\Lambda$ is \emph{locally convex} if for all $v \in
\Lambda^0,$ all $i,j \in \{1,\dots k\}$ with $i \ne j$, all $\lambda \in v\Lambda^{e_i}$ and all $\mu \in v\Lambda^{e_j}$, the sets $s(\lambda)\Lambda^{e_j}$ and $s(\mu)\Lambda^{e_i}$ are non-empty. Roughly speaking, local convexity stipulates that $\Lambda$ contains no subgraph resembling:
\begin{center}
\begin{tikzpicture}[>=stealth,scale=0.7]
    \node (u) at (0,2) {$u$};
    \node (v) at (0,0) {$v$}
        edge[<-,dashed] node[auto] {$\mu$} (u);
    \node (w) at (2,0) {$w$}
        edge[->] node[auto] {$\lambda$} (v);
\end{tikzpicture}
\end{center}
\end{defn}

\begin{defn}\label{defn FE}
For $v \in \Lambda^0$, a subset $E \subset v\Lambda$ is \emph{exhaustive} if for every $\mu \in v\Lambda$ there exists a $\lambda \in E$ such that $\Lmin (\lambda,\mu) \neq \emptyset$. We denote the set of all \emph{finite exhaustive subsets} of $\Lambda$ by $\FE(\Lambda)$.
\end{defn}

\begin{defn}\label{defn boundary paths}
An element $x \in W$ is a \emph{boundary path} if for all $n \in
\NN^k$ with $n \leq d(x)$ and for all $E \in x(n)\FE(\Lambda)$ there exists $m \in \NN^k$ such that
$x(n,m) \in E$. We write $\partial \Lambda$ for the set of all boundary paths.

Define the set $\Lambda^{\leq\infty}$ as follows. A $k$-graph morphism $x:\Omega_{k,m} \to \Lambda$ is an element of $\Lambda^{\leq\infty}$ if there exists $n_x \leq d(x)$ such that for $n \in \NN^k$ satisfying $n_x\leq n \leq d(x)$ and $n_i = d(x)_i$, we have $x(n)\Lambda^{e_i} = \emptyset.$
\end{defn}

\begin{rmk}\label{boudary path remark}
Raeburn, Sims and Yeend introduced $\Lambda^{\leq\infty}$ to construct a nonzero Cuntz-Krieger $\Lambda$-family \cite[Proposition 2.12]{RSY2004}. Farthing, Muhly and Yeend introduced $\partial \Lambda$ in \cite{FMY2005}; in order to construct a groupoid to which Renault's theory of groupoid $C^*$-algebras \cite{Renault1980} applied, they required a path space which was locally compact and Hausdorff in an appropriate topology, and $\Lambda^{\leq\infty}$ did not suffice. The differences between $\partial\Lambda$ and $\Lambda^{\leq\infty}$ can be easily seen if $\Lambda$ contains any infinite receivers (e.g. any path in a $1$-graph $\Lambda$ with source an infinite receiver is an element of $\partial\Lambda \setminus \Lambda^{\leq\infty}$), but can even show itself in the row-finite case if $\Lambda$ is not locally convex.
\end{rmk}

\begin{example}\label{boundary different to leq infty}
Suppose $\Lambda$ is the $2$-graph with the skeleton pictured below.
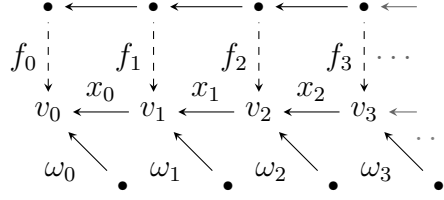
\begin{center}
\begin{tikzpicture}[>=stealth,scale=0.7]
    \pgfmathparse{sqrt(2)}
    \foreach \x in {0,1,2,3}
        {
        \node (u\x) at (\x*2,2) {$\scriptstyle \bullet$};
        \node (v\x) at (\x*2,0) {$v_{\x}$};
        \node (w\x) at (\x*2+\pgfmathresult,-\pgfmathresult) {$\scriptstyle \bullet$};
        }
    \foreach \x / \y in {0/1,1/2,2/3}
        {
        \draw[<-] (u\x) to (u\y);
        \draw[<-] (v\x) to node[auto,black] {$x_\x$} (v\y);
        \draw[<-] (v\x) to node[auto,swap,black] {$\omega_\x$} (w\x);
        \draw[<-,dashed] (v\x) to node[auto,black] {$f_\x$} (u\x);
        }
    \draw[<-,dashed] (v3) to node[auto,black] {$f_3$} node[auto,swap,black!60] {$\dots$} (u3);
    \draw[<-] (v3) to node[auto,black!50] {$\dots$} node[auto,swap,black] {$\omega_3$} (w3);
    \draw[<-,black!60] (v3) to (7,0);
    \draw[<-,black!60] (u3) to (7,2);
\end{tikzpicture}
\end{center}

Consider the paths $x = x_0 x_1 \dots$, and $\omega^{n} = x_0 x_1 \dots x_{n-1} \omega_n$ for $n=0,1,2,\dots$. Observe that $x \notin \Lambda^{\leq\infty}$: for each $n \in \NN$, we have $d(x)_2 = 0= (n,0)_2$, and $x((n,0)) \Lambda^{e_2}  = v_n \Lambda^{e_2} \ne \emptyset$.

We claim that $x \in \partial \Lambda$. Fix $m \in \NN$ and $E\in v_m\FE(\Lambda)$. Since $E$ is exhaustive, for each $n \geq m$, there exists $\lambda^n \in E$ such that $\MCE(\lambda^n,x_m\dots x_{n-1}\omega_n)\neq \emptyset$. Since $E$ is finite, it can not contain $x_m\dots x_{n-1}\omega_n$ for every $n \geq m$, so it must contain $x_m \dots x_p$ for some $p \in \NN$. So $x((m,0),(m+p)) = x_m \dots x_p$ belongs to $E$.

The $2$-graph of Example \ref{boundary different to leq infty} first appeared in Robertson's honours thesis \cite{dave} to illustrate a subtlety arising in Farthing's procedure \cite{Farthing2008} for removing sources in $k$-graphs when the $k$-graphs in question are not locally convex. It was for this reason that only locally convex $k$-graphs in the main results of \cite{dave,RS2009}.
\end{example}

\begin{prop}\label{leqsubsetpartial}
Suppose $\Lambda$ is a finitely aligned $k$-graph. Then $\Lambda^{\leq \infty} \subset \partial \Lambda$. If $\Lambda$ is row-finite and locally convex, then $\Lambda^{\leq\infty}=\partial\Lambda$.
\end{prop}

To prove this we use the following lemma.

\begin{lemma}\label{v lambda ei non empty then exh}
Let $\Lambda$ be a row-finite, locally convex $k$-graph, and suppose that $v \in \Lambda^0$ satisfies
$v\Lambda^{e_i} \ne \emptyset$ for some $i \leq k$. Then $v\Lambda^{e_i} \in v \FE(\Lambda)$.
\begin{proof}
Since $\Lambda$ is row-finite, $v \Lambda^{e_i}$ is finite. To see that it is exhaustive, let $\mu \in
v\Lambda$. If $d(\mu)_i > 0$, then $g = \mu(0,e_i) \in v\Lambda^{e_i}$ implies that $\Lmin(\mu,g) \ne
\emptyset$. Suppose that $d(\mu)_i = 0$. Let $\mu = \mu_1 \dots \mu_n$ be a factorisation of
$\mu$ such that $|d(\mu_j)|=1$ for each $j \leq n$. Since $\Lambda$ is locally convex, $s(\mu)\Lambda^{e_i} = s(\mu_n)\Lambda^{e_i} \ne \emptyset$. Fix $g \in s(\mu)\Lambda^{e_i}$. Let $f:=(\mu g)(0,e_i)$. Then $f \in v\Lambda^{e_i}$. Since $d(\mu_i)=0$, we have $d(\mu g) = d(\mu) \vee d(f)$. Hence $(g,(\mu g)(e_i,d(\mu g))) \in \Lambda^{\min}(\mu,f)$ as required.
\end{proof}
\end{lemma}

\begin{proof}[Proof of Proposition~\ref{leqsubsetpartial}]
Fix $x \in \Lambda^{\leq\infty}$. Let $m \leq d(x)$ and $E \in x(m)\FE(\Lambda)$. Define $t \in \NN^k$ by
\[
    t_i :=
        \begin{cases}
            d(x)_i &\text{if $d(x)_i < \infty$}, \\
            \ds\max_{\lambda \in E} \big(n_x \vee (m + d(\lambda))\big)_i &\text{if $d(x)_i=\infty$}.
        \end{cases}
\]
Then $x(m,t) \in x(m)\Lambda$, so there exists $\lambda \in E$ such that $\Lmin(x(m,t),\lambda)$ is non-empty. Let $(\alpha,\beta) \in \Lmin(x(m,t),\lambda)$. We first show that $d(\alpha)=0$. Since $x \in \Lambda^{\leq \infty}$ and $n_x \leq t \leq d(x)$, if $d(x)_i < \infty$ then $x(t)\Lambda^{e_i} = \emptyset$. So for each $i$ such that $d(x)_i < \infty$, we have $d(\alpha)_i = 0$. Now suppose that $d(x)_i = \infty$. Then $d(x(m,t))_i = t_i -m_i \geq d(\lambda)_i$. So $d(x(m,t)\alpha)_i = \max\{d(x(m,t))_i, d(\lambda)_i\} = d(x(m,t))_i$, giving $d(\alpha)_i = 0$. Then we have $x(m,t) = \lambda\beta$, so $x(m,m+d(\lambda)) = \lambda$.

Now suppose that $\Lambda$ is row-finite and locally convex. We want to show $\partial \Lambda \subset
\Lambda^{\leq\infty}$. Fix $x \in \partial \Lambda$, and $n \in \NN^k$ such that $n \leq d(x)$ and $n_i = d(x)_i$. It suffices to show that $x(n)\Lambda^{e_i} = \emptyset$. Since $n_i = d(x)_i$, we have $x(n)\Lambda^{e_i} \notin x(n)\FE(\Lambda)$. Lemma~\ref{v lambda ei non empty then exh} then implies that $x(n)\Lambda^{e_i} = \emptyset$.
\end{proof}

\section{Path Space Topology}\label{hrg top}
Following the approach of Paterson and Welch in \cite{PW2005}, we construct a locally compact Hausdorff topology on the path space $W$ of a finitely aligned $k$-graph $\Lambda$. The \emph{cylinder set of $\mu \in \Lambda$} is $\Zz(\mu) := \{ \nu \in W : \nu(0,d(\mu)) = \mu \}$. Define $\alpha: W \to\{0,1\} ^{\Lambda}$ by $\alpha(w)(y) = 1$ if $w \in \Zz(y)$ and $0$ otherwise. For a finite subset $G \subset s(\mu)\Lambda$ we define
\begin{equation}%\label{basissetdefn}
    \Zz(\mu\setminus G):= \Zz(\mu) \setminus \bigcup_{\nu \in G} \Zz(\mu\nu).
\end{equation}
Our goals for this section are the following two theorems. The basis we end up with is slightly different to that in \cite[Corollary 2.4]{PW2005}, revealing a minor oversight of the authors.

\begin{theorem}\label{hrg topology}
  Let $\Lambda$ be a finitely aligned $k$-graph. Then the collection
\[
 \Big\{ \Zz(\mu \setminus G) : \mu \in \Lambda \text{ and } G\subset \bigcup_{i=1}^k(s(\mu)\Lambda^{e_i})\text{ is finite} \Big\}
\]
is a base for the initial topology on $W$ induced by $\{\alpha\}$.
\end{theorem}

\begin{theorem}\label{hrg lchs}
Let $\Lambda$ be a finitely aligned higher-rank graph. With the topology described in Theorem~\ref{hrg topology}, W is a locally compact Hausdorff space.
\end{theorem}

Let $F$ be a set of paths in a $k$-graph $\Lambda$. A path $\beta \in W$ is a \emph{common extension of the paths in $F$} if for each $\mu \in F$, we can write $\beta = \mu \beta_\mu$ for some $\beta_\mu \in W$. If in addition $d(\beta) = \bigvee_{\mu\in F} d(\mu)$, then $\beta$ is a \emph{minimal common extension of the paths in $F$}. We denote the set of all minimal common extensions of the paths in $F$ by $\MCE(F)$. Since $\MCE(\{\mu,\nu\}) = \MCE(\mu,\nu)$, this definition is consistent with Definition~\ref{defn:lmin and mce}.

\begin{rmk}
If $F \subset \Lambda$ is finite, then
$\bigcap_{\mu\in F} \Zz(\mu) = \bigcup_{\beta \in \MCE(F)} \Zz(\beta).$
\end{rmk}

\begin{proof}[Proof of Theorem~\ref{hrg topology}]
We first describe the topology on $\{0,1\} ^\Lambda$. Given disjoint finite subsets $F,G \subset \Lambda$ and $\mu\in\Lambda$, define sets $U_\mu^{F,G}$  to be $\{1\}$ if $\mu \in F$, $\{0\}$ if $\mu \in G$ and $\{0,1\}$ otherwise. Then the sets $N(F, G):= \prod_{\mu\in\Lambda}U_\mu^{F,G}$ where $F,G$ range over all finite disjoint pairs of subsets of $\Lambda$ form a base for the topology on $\{0,1\}^{\Lambda}$.

Clearly, $\alpha$ is a homeomorphism onto its range, and hence the sets $\alpha^{-1}(N(F, G))$ are a base for a topology on $W$. Routine calculation shows that
\[
\alpha^{-1}(N(F, G))= \left( \bigcup_{\mu \in \MCE(F)} \Zz(\mu) \right) \setminus \left( \bigcup_{\nu \in G}
\Zz(\nu)\right),
\]
so the sets
$
\Zz(\mu) \setminus \bigcup_{\nu \in G} \Zz(\mu\nu) = \Zz( \mu \setminus G)
$
are a base for our topology.

To finish the proof, it suffices to show that for $\mu \in \Lambda$, a finite subset $G \subset s(\mu)\Lambda$ and $\lambda \in \Zz(\mu \setminus G)$, there exist $\alpha \in \Lambda$ and a finite $F \subset \bigcup_{i=1}^k\big(s(\alpha)\Lambda^{e_i}\big)$ such that
$
\lambda \in \Zz(\alpha \setminus F) \subset \Zz(\mu \setminus G).
$
Let $N:= \left(\bigvee_{\nu \in G} d(\mu\nu)\right) \wedge d(\lambda)$ and $\alpha =
\lambda(0,N)$. To define $F$, we first define a set $F_\nu$ associated to each $\nu \in G$, then take
$F=\bigcup_{\nu \in G} F_\nu$. Fix $\nu \in G$. We consider the following cases:
\begin{enumerate}
\item If $N \geq d(\mu\nu)$ or $\MCE(\alpha,\mu\nu) = \emptyset$, let $F_\nu = \emptyset$.
\item If $N \ngeq d(\mu\nu)$ and $\MCE(\alpha,\mu\nu) \neq \emptyset$, define $F_\nu$ as follows:
\end{enumerate}
Since $N \ngeq d(\mu\nu)$, there exists $j_\nu \leq k$ such that $N_{j_\nu} < d(\mu\nu)_{j_\nu}$. Hence each $\gamma \in \MCE(\alpha,\mu\nu)$ satisfies $d(\gamma)_{j_\nu} = (N \vee d(\mu\nu))_{j_\nu} > N_{j_\nu}$. Define $F_\nu = \{ \gamma(N, N + e_{j_\nu}): \gamma \in \MCE(\alpha,\mu\nu)\}.$ Since $\Lambda$ is finitely aligned, $F_\nu$ is finite.

We now show that $\lambda \in \Zz(\alpha \setminus F)$. We have $\lambda \in \Zz(\alpha)$ by choice of
$\alpha$. If $F = \emptyset$ we are done. If not, then fix $\nu \in G$ such that $F_\nu \ne \emptyset$, and fix $e \in F_\nu$. Then $e = \gamma(N,N+ e_{j_\nu})$ for some $\gamma \in \MCE(\alpha,\mu\nu)$. Then $d(\lambda)_{j_\nu} = N_{j_\nu} < (N+ e_{j_\nu})_{j_\nu} = d(\alpha e)_{j_\nu}$. So $\lambda \notin \Zz(\alpha e)$, hence $\lambda \in \Zz(\alpha \setminus F)$.

We now show that $\Zz(\alpha \setminus F) \subset \Zz(\mu \setminus G)$. Fix $\beta \in \Zz(\alpha \setminus F)$. Since $\alpha \in \Zz(\mu)$, we have $\beta \in \Zz(\mu)$. Fix $\nu \in G$. We show that $\beta \notin \Zz(\mu\nu)$ in cases:
\begin{enumerate}
\item Suppose that $N \geq d(\mu\nu)$. Since $\beta \in \Zz(\alpha) = \Zz(\lambda(0,N))$ and $\lambda \notin \Zz(\mu\nu)$, it follows that $\beta \notin \Zz(\mu\nu)$.
\item If $N \ngeq d(\mu\nu)$, then either
\begin{enumerate}
\item $\MCE(\alpha,\mu\nu) = \emptyset$, in which case $\beta \in \Zz(\alpha)$ implies that $\beta \notin \Zz(\mu\nu)$; or
\item $\MCE(\alpha,\mu\nu) \neq \emptyset$. Then for each $\gamma \in \MCE(\alpha,\mu\nu)$, we know $\beta(N,N+e_{j_\nu}) \neq \gamma(N,N+e_{j_\nu})$. It then follows that $\beta \notin \Zz(\mu\nu).$\qedhere
\end{enumerate}
\end{enumerate}
%Since this holds for all $\nu \in G$, we have $\beta \in \Zz(\mu \setminus G)$.
\end{proof}

%For the proof of Theorem~\ref{hrg lchs}, we use the following technical results.

\begin{lemma}\label{nested_paths}
Let $\{\nu^{(n)}\}$ be a sequence of paths in $\Lambda$ such that \begin{itemize}
\item[(i)] $d(\nu^{(n+1)}) \geq d(\nu^{(n)})$ for all $n \in \NN$, and
\item[(ii)] $\nu^{(n+1)}\left(0,d(\nu^{(n)})\right) = \nu^{(n)}$ for all $n \in \NN$.
\end{itemize}
Then there exists a unique $\omega \in W$ with $d(\omega) = \bigvee_{n\in\NN} d(\nu^{(n)})$ and
$\omega\left(0,d(\nu^{(n)})\right) = \nu^{(n)}$ for all $n \in \NN$.
\begin{proof}
 Let $m = \bigvee_{n\in\NN} d(\nu^{(n)}) \in (\NN \cup \{\infty\})^k.$ Then
\begin{equation}\label{there exists Na}\parbox{0.8\textwidth}{For $a\in \NN^k$ with $a \leq m$, there exists $N_a \in \NN$ such that $d(\nu^{(N_a)}) \geq a$.}\end{equation}

For each $(p,q) \in \Omega_{k,m}$ apply \eqref{there exists Na} with $a=q$ and define $\omega(p,q) = \nu^{(N_q)}(p,q)$. Routine calculations using \eqref{there exists Na} show that $\omega:\Omega_{k,m} \to \Lambda$ is a well-defined graph morphism with the required properties.
\end{proof}
\end{lemma}

\begin{proof}[Proof of Theorem~\ref{hrg lchs}]
Fix $v \in \Lambda^0$. We follow the strategy of \cite[Theorem 2.2]{PW2005} to show $\Zz(v)$ is compact: since $\alpha$ is a homeomorphism onto its range, and since $\{ 0,1 \}^\Lambda$ is compact, it suffices to prove that $\alpha(\Zz(v))$ is closed in $\{ 0,1 \}^\Lambda$. Suppose that $(\omega^{(n)})_{n \in \NN }$ is a sequence in $\Zz(v)$ such that converging to $f \in \{0,1\}^\Lambda$. We seek $\omega \in \Zz(v)$ such that $f = \alpha(\omega)$. Define $A=\{\nu\in\Lambda : \alpha(\omega^{(n)})(\nu) \to 1 \text{ as } n \to \infty \}.$ Then $A\neq \emptyset$ since $v \in A$. Let $d(A) := \bigvee_{\nu \in A} d(\nu)$.
\setcounter{theorem}{2}
\setcounter{claim}{0}
\begin{claim}\label{claim: there exists omega such that}
There exists $\omega \in v\Lambda^{d(A)}$ such that:
\begin{itemize}
\item $d(\omega) \geq d(\mu)$ for all $\mu \in A$, and
\item $\omega(0,n) \in A$ for all $n \in \NN^k$ with  $n\leq d(A)$.
\end{itemize}
\begin{proof}
To define $\omega$ we construct a sequence of paths and apply Lemma~\ref{nested_paths}. We first show that for each pair $\mu,\nu \in A$, $\MCE(\mu,\nu)\cap A$ contains exactly one element. Fix $\mu,\nu \in A$. Then for large enough $n$, there exist $\beta^{n} \in \MCE(\mu,\nu)$ such that $\omega^{n} = \beta^{n}(\omega^{n})'.$ Since $\MCE(\mu,\nu)$ is finite, there exists $M$ such that $\omega^{n} = \beta^{M}(\omega^{n})'$ for infinitely many $n$. Define $\beta_{\mu,\nu} := \beta^{M}$. Then $\beta_{\mu,\nu} \in A$. For uniqueness, suppose that $\phi \in \MCE(\mu,\nu)\cap A$. Then for large $n$ we have
$
\beta_{\mu,\nu} = \omega^n(0,d(\mu)\vee d(\nu)) = \phi.
$

Since $A$ is countable, we can list $A = \{ \nu^{1}, \nu^{2}, \dots, \nu^{m}, \dots \}.$ Let $y^{1} := \nu^{1}$, and iteratively define $y^{n} = \beta_{y^{n-1},\nu^{n}}$. Then
$
    d(y^{n}) = d(y^{n-1}) \vee d(\nu^{n}) \geq d(y^{n-1}),
$
and $y^{n}(0,y^{n-1}) = y^{n-1}$. By Lemma~\ref{nested_paths}, there exists a unique $\omega \in W$ satisfying $d(\omega) = d(A)$ and $\omega(0,d(y^{n})) = y^{n}$ for all $n$. It then follows from \eqref{there exists Na} that $\omega(0,n) \in A$ for all $n \leq d(A)$.\pc\end{proof}
\end{claim}

To see $\alpha(\Zz(v))$ is closed, fix $\lambda \in \Lambda$. We show that $\alpha(\omega^{(n)})(\lambda) \to \alpha(\omega)(\lambda)$. If $\alpha(\omega)(\lambda) = 1$, then $\lambda = \omega(0,d(\lambda)) \in A$ by Claim~\ref{claim: there exists omega such that}, and thus $\alpha(\omega^{(n)})(\lambda) \to 1$ as $n \to \infty$. Now suppose that $\alpha(\omega)(\lambda) = 0$. If $d(\lambda) \nleq d(\omega)$, then $\lambda \notin A$ by Claim~\ref{claim: there exists omega such that}, forcing $\alpha(\omega^{(n)})(\lambda) \to 0$. Suppose that $d(\lambda) \leq d(\omega)$. Since $\omega(0,d(\lambda)) \in A$, we have $\omega^{(n)}(0,d(\lambda)) = \omega(0,d(\lambda))$ for large $n$. Then $\alpha(\omega)(\lambda) = 0$ implies that $\omega(0,d(\lambda)) \neq \lambda$. So for large enough $n$ we have $\omega^{(n)}(0,d(\lambda)) \neq \lambda$, forcing $\alpha(\omega^{(n)})(\lambda) \to 0$.
\end{proof}

\setcounter{theorem}{4}
\begin{rmk}\label{lambda leq infty not closed}
It has been shown that $\partial\Lambda$ is a closed subset of $W$~\cite[Lemma 5.12]{FMY2005}. Hence $\partial\Lambda$, with the relative topology, is a locally compact Hausdorff space. Consider the $2$-graph of Example~\ref{boundary different to leq infty}. For each $n\in\NN$, we have $\omega^n \in \Lambda^{\leq\infty}$. Notice that $\omega^n \to x \notin \Lambda^{\leq\infty}$. So $\Lambda^{\leq \infty}$ is not closed in general, and hence is not locally compact.
\end{rmk}

\section{Removing Sources}\label{hrg removing sources}

\begin{theorem}\label{finite aligndness preserved}
Let $\Lambda$ be a finitely aligned $k$-graph. Then there is a finitely aligned $k$-graph $\wt\Lambda$ with no sources, and an embedding $\iota$ of $\Lambda$ in $\wt\Lambda$. If $\Lambda$ is row-finite, then so is $\wt\Lambda$.
\end{theorem}

\begin{defn}
Define a relation $\approx$ on $V_\Lambda:= \{(x;m) : x \in \partial\Lambda, m \in \NN^k\}$ by: $(x;m)
\approx (y;p)$ if and only if
\renewcommand{\theenumi}{V\arabic{enumi}}
\begin{enumerate}
\item $x(m \wedge d(x)) = y(p \wedge d(y))$; and\label{v1}
\item $m - m\wedge d(x) = p - p \wedge d(y)$.\label{v2}
\end{enumerate}
\end{defn}

\begin{defn}
Define a relation $\sim$ on $P_\Lambda:= \{(x;(m,n)) : x \in \partial\Lambda, m\leq n \in \NN^k\}$ by:
$(x;(m,n)) \sim (y;(p,q))$ if and only if
\renewcommand{\theenumi}{P\arabic{enumi}}
\begin{enumerate}
\item $x(m \wedge d(x),n \wedge d(x)) = y(p \wedge d(y), q \wedge d(y))$;\label{p1}
\item $m - m\wedge d(x) = p - p \wedge d(y)$; and\label{p2}
\item $n-m = q-p$.\label{p3}
\end{enumerate}
\end{defn}

It is clear from their definitions that both $\approx$ and $\sim$ are equivalence relations.

\begin{lemma}\label{other version of p2}
Suppose that $(x;(m,n)) \sim (y;(p,q))$. Then $n-n\wedge d(x) = q-q \wedge d(y)$.
\begin{proof}
  It follows from \eqref{p1} and \eqref{p3} that
  \[
  n- n \wedge d(x) - (m-m\wedge d(x)) = q-q\wedge d(y) - (p - p\wedge d(y)).
  \]
  The result then follows from \eqref{p2}.
\end{proof}
\end{lemma}

Let $\wt{P_\Lambda}:= P_\Lambda / \sim$ and $\wt{V_\Lambda}:= V_\Lambda / \approx$. The class in $\wt{P_\Lambda}$ of $(x;(m,n)) \in P_\Lambda$ is denoted $[x;(m,n)]$, and similarly the class in $\wt{V_\Lambda}$ of $(x;m) \in V_\Lambda$ is denoted $[x;m]$.

To define the range and source maps, observe that if $(x;(m,n)) \sim (y;(p,q))$, then $(x;m) \approx (y;p)$ by definition, and $(x;n) \approx (y;q)$ by Lemma~\ref{other version of p2}. We define range and source maps as follows.

\begin{defn}
Define $\wt r, \wt s: \wt{P_\Lambda} \to \wt{V_\Lambda}$ by:
\[
    \wt r([x;(m,n)]) = [x;m] \qquad\text{and}\qquad \wt s([x;(m,n)]) = [x,n].
\]
\end{defn}

We now define composition. For each $m \in \NN^k$, we define the \emph{shift map} $\sigma^m:\bigcup_{n\geq m}\Lambda^n \to \Lambda$ by $\sigma^m(\lambda)(p,q) = \lambda(p+m,q+m)$.

\begin{prop}\label{need this for comp}
Suppose that $\Lambda$ is a $k$-graph and let $[x;(m,n)]$ and $[y;(p,q)]$ be elements of $\wt{P_\Lambda}$ satisfying $[x;n] = [y;p]$. Let $z:=x(0,n\wedge d(x)) \sigma^{p \wedge d(y)}y.$ Then

\begin{enumerate}
\item $z \in \partial\Lambda$;
\item $m \wedge d(x) = m \wedge d(z)$ and $n \wedge d(x) = n \wedge d(z)$;
\item $x(m\wedge d(x),n\wedge d(x)) = z(m\wedge d(z), n \wedge d(z))$ and $y(p\wedge d(y),q\wedge d(y))=
z(n \wedge d(z), (n+q-p) \wedge d(z)). $
\end{enumerate}
\begin{proof}
Part (1) follows from \cite[Lemma 5.13]{FMY2005}, and (2) and (3) can be proved as in \cite[Proposition 2.11]{Farthing2008}.
\end{proof}
\end{prop}

Fix $[x;(m,n)],[y;(p,q)] \in \wt{P_\Lambda}$ such that $[x;n] = [y;p]$, and let $z = x(0,n\wedge d(x)) \sigma^{p \wedge d(y)}y.$  That the formula
\begin{equation}\label{comp on desource}
    [x;(m,n)] \circ [y;(p,q)] = [z;(m,n+q-p)]
\end{equation}
determines a well-defined composition follows from Proposition~\ref{need this for comp}.

Define $\id:\wt{V_\Lambda} \to \wt{P_\Lambda}$ by $\id_{[x;m]} = [x;(m,m)]$.

\begin{prop}[{\cite[Lemma 2.19]{Farthing2008}}]\label{Lambda tilde is a category}
$\wt{\Lambda}:=(\wt{V_\Lambda}, \wt{P_\Lambda}, \wt r, \wt s, \circ, \id)$ is a category.
\end{prop}

%We view $\NN^k$ as a category with $\Obj(\NN^k)= \{\star\}$, $\Mor(\NN^k) = \NN^k$ and with composition defined by addition.
\begin{defn}
Define $\wt d:\wt \Lambda \to \NN^k$ by $\wt d(v) = \star$ for all $v \in \wt{V_\Lambda}$, and $\wt d([x;(m,n)]) = n-m$ for all $[x;(m,n)] \in \wt{P_\Lambda}.$
\end{defn}

\begin{prop}[{\cite[Theorem 2.22]{Farthing2008}}]\label{facprop}
The map $\wt d$ defined above satisfies the factorisation property. Hence with $\wt \Lambda$ as in Proposition~\ref{Lambda tilde is a category}, $(\wt\Lambda,\wt d)$ is a $k$-graph with no sources.
\end{prop}

\begin{example}\label{desource ex}
If we allow infinite receivers, our construction yields a different $k$-graph to Farthing's construction in \cite[\S 2]{Farthing2008}: consider the 1-graph $E$ with an infinite number of loops $f_i$ on a single vertex $v$:

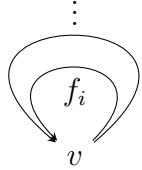
\begin{center}
\begin{tikzpicture}[>=stealth,scale=0.7]
    \node (v) at (0,0) {$v$};
    \draw[ ->,loop, looseness = 10] (v) to node[auto] {$f_i$} (v);
    \draw[ ->,loop, looseness = 15, in=140, out = 40] (v) to node[auto,swap,black] {$\vdots$} (v);
\end{tikzpicture}
\end{center}

Here we have $E^{\leq \infty} = \emptyset$, so Farthing's construction yields a $1$-graph $\overline E \cong E$. Since $v$ belongs to every finite exhaustive set in $E$, we have $\partial E = E$. Furthermore $[f_j;p] = [f_i;p] = [v;p]$ for all $i,j,p \in \NN$, and
\[
    [f_j;(p,q)] = [f_i;(p,q)] = [v;(p-1,q-1)]
\]
for all $i,j,p,q$ such that $1 < p \leq q$. Thus there is exactly one path between any two of the added vertices, resulting in a head at $v$, yielding the graph illustrated below

 
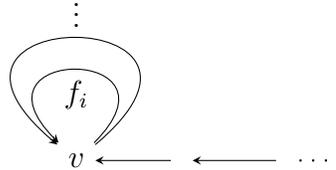
\begin{center}
\begin{tikzpicture}[>=stealth,scale=0.7]
    \node (v) at (0,0) {$v$};
    \draw[ ->,loop, looseness = 10] (v) to node[auto] {$f_i$} (v);
    \draw[ ->,loop, looseness = 15, in=140, out = 40] (v) to node[auto,swap,black] {$\vdots$} (v);
    \node (v1) at (2,0) {}
        edge[->] (v);
    \node (v2) at (4,0) {}
        edge[->] (v1);
    \draw[loosely dotted, thick] (5,0) -- (v2);
\end{tikzpicture}
\end{center}

It is intriguing that following Drinen and Tomforde's desingularisation, a head is also added at infinite receivers like this, and then the ranges of the edges $f_i$ are distributed along this head --- we cannot help but wonder whether this might suggest an approach to a Drinen-Tomforde desingularisation for $k$-graphs.
\end{example}

\subsection{Row-finite $1$-graphs.}
While one expects this style of desourcification to agree with adding heads to a row-finite $1$-graph as in \cite{BPRS2000}, this appears not to have been checked anywhere.

\begin{prop}Let $E$ be a row-finite directed graph and $F$ be the graph obtained by adding heads to sources, as in \cite[p4]{BPRS2000}. Let $\Lambda$ be the $1$-graph associated to $E$. Then $\wt\Lambda \cong F^*$, where $F^*$ is a the path-category of $F$.
\begin{proof}
Define $\eta': P_\Lambda \to F^*$ as follows. Fix $x \in \partial E$ and $m,n \in \NN$. Then either $x \in E^\infty$, or $x \in E^*$ and $s(x)$ is a source in $E$. If $x \in E^\infty$, define $\eta'((x;(m,n))) = x(m,n)$. For $x \in E^*$, let $\mu_x$ be the head added to $s(x)$, and define $\eta'((x;(m,n))) = (x \mu_x)(m,n)$. It is straightforward to check that $\eta'$ respects the equivalence relation $\sim$ on $P_\Lambda$. Define $\eta:\wt\Lambda \to F^*$ by $\eta([x;(m,n)]) = \eta'((x;(m,n))).$ Easy but tedious calculations show that $\eta$ is a graph morphism.

We now construct a graph morphism $\xi:F^* \to \wt\Lambda$. Let $\nu \in F^*$. To define $\xi$ we first need some preliminary notation. $\xi$ will be defined casewise, broken up as follows:
\renewcommand{\theenumi}{\roman{enumi}}
\begin{enumerate}
\item $\nu \in E^*$,
\item $r(\nu)\in E^*$ and $s(\nu) \in F^* \setminus E^*$, or
\item $r(\nu),s(\nu) \in F^* \setminus E^*$.
\end{enumerate}

If $\nu \in E^*$, fix $\alpha_\nu \in s(\nu)\partial E$. If $\nu$ has $r(\nu) \in E^*$ and $s(\nu) \in F^* \setminus E^*$, let $p_\nu = \max \{p \in \NN : \nu(0,p) \in E^*\}$. Then $\nu(p_\nu)$ is a source in $E^*$, and $\nu(0,p_\nu) \in \partial E$. If $\nu \in F^* \setminus E^*$, then $\nu$ is a segment of a head $\mu_\nu$ added to a source in $E^*$, and we let $q_\nu$ be such that $\nu = \mu_\nu(q_\nu, q_\nu + d(\mu))$.

We then define $\xi$ by
\[
\xi(\nu) =
\begin{cases}
    [\nu\alpha_\nu;(0,d(\nu))] &\text{if $\nu \in E^*$}\\
    [\nu(0,p_\nu);(0,d(\nu))] &\text{if $r(\nu)\in E^*$ and $s(\nu) \notin E^*$}\\
    [r(\mu_\nu);(q_\nu, q_\nu+d(\nu))] &\text{if $r(\nu),s(\nu) \in F^* \setminus E^*$}.
\end{cases}
\]

Again, tedious but straightforward calculations show that $\xi$ is a well-defined graph morphism, and that $\xi \circ \eta = 1_{\wt\Lambda}$ and $\eta \circ \xi = 1_{F^*}$.
\end{proof}
\end{prop}

When $\Lambda$ is row-finite and locally convex, Proposition~\ref{leqsubsetpartial} implies that $\Lambda^{\leq\infty} = \partial\Lambda$. In this case our construction is essentially the same as that of Farthing \cite[\S2]{Farthing2008}, with notation adopted as in \cite{RS2009}. If $\Lambda$ is row-finite but not locally convex, then $\Lambda^{\leq\infty} \subset \partial\Lambda$ (Example~\ref{boundary different to leq infty} shows that this may be a strict containment). Thus it is reasonable to suspect that our construction could result in a larger path space than Farthing's. Interestingly, this is not the case.

\begin{prop}\label{me and cindy agree when the hrg is rf and lc}
Let $\Lambda$ be a row-finite $k$-graph. Suppose that $x\in\partial\Lambda \setminus \Lambda^{\leq\infty}$ and $m \leq n \in \NN^k$. Then there exists $y \in \Lambda^{\leq\infty}$ such that $(x;(m,n)) \sim (y;(m,n))$.
 \begin{proof}
Since $x \notin \Lambda^{\leq\infty}$, there exists $q \geq n\wedge d(x)$ and $i \leq k$ such that $q \leq d(x)$, $q_i = d(x)_i$, and $x(q)\Lambda^{e_i}\neq \emptyset$. Let
\[
    J:= \{ i \leq k: q_i = d(x)_i \text{ and } x(q)\Lambda^{e_i} \neq \emptyset\}.
\]
Since $x \in \partial\Lambda$, for each $E \in x(q)\FE(\Lambda)$ there exists $t \in \NN^k$ such that $x(q, q+t) \in E$. Since $q_i = d(x)_i$ for all $i \in J$, the set $\bigcup_{i \in J}x(q)\Lambda^{e_i}$ contains no such segments of $x$, and thus cannot be finite exhaustive. Since $\Lambda$ is row-finite, $\bigcup_{i\in J}x(q) \Lambda^{e_i}$ is finite, so $\bigcup_{i\in J}x(q) \Lambda^{e_i}$ is not exhaustive. Thus there exists $\mu \in x(q)\Lambda$ such that $\MCE(\mu,\nu) = \emptyset$ for all $\nu \in \bigcup_{i \in J}x(q)\Lambda^{e_i}$. By \cite[Lemma 2.11]{RSY2004}, $s(\mu)\Lambda^{\leq\infty} \neq \emptyset$. Let $z \in s(\mu)\Lambda^{\leq\infty}$, and define $y:=x(0,q)\mu z$. Then $y \in \Lambda^{\leq\infty}$ by \cite[Lemma 2.10]{RSY2004}.

Now we show that $(x;(m,n)) \sim (y;(m,n))$. Condition~\eqref{p3} is trivially satisfied. To see that \eqref{p1} and \eqref{p2} hold, it suffices to show that $n\wedge d(x) = n \wedge d(y)$. Firstly, let $i \in J$. If $d(\mu z)_i \neq 0$, then $(\mu z)(0, d(\mu) + e_i) \in \MCE(\mu,\nu)$ for $\nu = (\mu z)(0,e_i) \in r(\mu)\Lambda^{e_i} = x(q) \Lambda^{e_i}$, a contradiction. So for each $i \in J$, $d(\mu z)_i = 0$, and hence $d(y)_i = d(x)_i$. Now suppose that $i \notin J$. Then either $x(q)\Lambda^{e_i} = \emptyset$ or $q_i < d(x)_i$. If $x(q)\Lambda^{e_i} = \emptyset$ then $d(y)_i = d(x)_i$. So suppose that $q_i < d(x)_i$. Since $n \wedge d(x) \leq q$, it follows that $n_i < d(x)_i$ and $n_i \leq q_i \leq d(y)_i$, hence $(n \wedge d(x))_i = n_i = (n \wedge d(y))_i.$ So $n \wedge d(x) = n \wedge d(y)$.
\end{proof}
\end{prop}

The following result allows us to identify $\Lambda$ with a subgraph of $\wt\Lambda$.

\begin{prop}\label{iotamorph2}
Suppose that $\Lambda$ is a $k$-graph, and that $\lambda \in \Lambda$. Then $s(\lambda)\partial\Lambda \neq \emptyset$. If $x,y \in s(\lambda)\partial\Lambda$, then $\lambda x, \lambda y \in \partial\Lambda$ and $(\lambda x;(0,d(\lambda))) \sim (\lambda y ;(0,d(\lambda)))$. Moreover, there is an injective $k$-graph morphism $\iota:\Lambda \to \wt\Lambda$ such that for $\lambda \in \Lambda$
\[
\iota(\lambda) = [ \lambda x ; (0, d(\lambda)) ] \text{ for any $x \in s(\lambda)\partial\Lambda$.}
\]
\begin{proof}
By \cite[Lemma 5.15]{FMY2005}, we have $v \partial\Lambda \neq \emptyset$ for all $v \in \Lambda^0$. In
particular, we have $s(\lambda) \partial\Lambda \neq \emptyset$. Let $x,y \in
s(\lambda)\partial\Lambda$. Then \cite[Lemma 5.13(ii)]{FMY2005} says that $\lambda x, \lambda y \in
\partial\Lambda$. It follows from the definition of $\sim$ that $(\lambda x;(0,d(\lambda))) \sim (\lambda y ;(0,d(\lambda)))$. Then straightforward calculations show that that $\iota$ is an injective $k$-graph morphism.
\end{proof}
\end{prop}

%\begin{lemma}
%Let $\Lambda$ be a finitely aligned $k$-graph, and let $\lambda, \mu \in \Lambda$. Then $\iota(\Lambda^{\min}(\lambda,\mu)) = \wt\Lambda^{\min}(\iota(\lambda),\iota(\mu))$.
%\begin{proof}
%Let $(\nu,\omega) \in \Lambda^{\min}(\lambda,\mu)$. Then since $\iota$ is a $k$-graph morphism, we have $(\iota(\nu),\iota(\omega)) \in
%\end{proof}
%\end{lemma}

We want to extend $\iota$ to an injection of $W_\Lambda$ into $W_{\wt\Lambda}$. The next proposition shows that any injective $k$-graph morphism defined on $\Lambda$ can be extended to $W_\Lambda$.

\begin{prop}\label{extend kgm to nonfinite paths}
Let $\Lambda,\Gamma$ be $k$-graphs and $\phi:\Lambda \to \Gamma$ be a $k$-graph morphism. Let $x \in
W_\Lambda \setminus \Lambda$, then $\phi(x):\Omega_{k,d(x)} \to W_\Gamma$ defined by $\phi(x)(p,q) =
\phi(x(p,q))$ belongs to $W_\Gamma$.
\begin{proof}
Follows from $\phi$ being a $k$-graph morphism.
\end{proof}
\end{prop}

In particular, we can extend $\iota$ to paths with non-finite degree. We need to know that composition works as expected for non-finite paths.

\begin{prop}
Let $\Lambda, \Gamma$ be $k$-graphs and $\phi:\Lambda \to \Gamma$ be a $k$-graph morphism. Let $\lambda
\in \Lambda$, $x \in s(\lambda)W_\Lambda$, and suppose that $n \in \NN^k$ satisfies $n \leq d(x)$. Then
\begin{enumerate}
\item $\phi(\lambda)\phi(x) = \phi(\lambda x)$; and
\item $\sigma^n(\phi(x)) = \phi(\sigma^n(x)).$
\end{enumerate}
\begin{proof} Follows from $\phi$ being a $k$-graph morphism.\end{proof}
\end{prop}

\begin{rmk}
It follows that the extension of an injective $k$-graph morphism to $W_\Lambda$ is also injective. In particular, the map $\iota:\Lambda \to \wt\Lambda$ has an injective extension $\iota:W_\Lambda \to W_{\wt\Lambda}$.
\end{rmk}

We need to be able to `project' paths from $\wt\Lambda$ onto the embedding $\iota(\Lambda)$ of $\Lambda$. For $y \in \partial \Lambda$ define
\begin{equation}\label{finite path hrg proj}
\pi([y;(m,n)]) = [y;(m \wedge d(y), n \wedge d(y))].
\end{equation}
Straightforward calculations show that $\pi$ is a surjective functor, and is a projection in the sense that $\pi(\pi([y;(m,n)])) = \pi([y;(m,n)])$ for all $[y;(m,n)] \in \wt\Lambda$. In particular, $\pi|_{\iota(\Lambda)} = \id_{\iota(\Lambda)}$.

\begin{lemma}\label{faprooflem1}
Let $\Lambda$ be a $k$-graph. Suppose that $\lambda,\mu \in \wt \Lambda$, and that $\lambda \in \Zz(\mu)$. Then $\pi(\lambda) \in \Zz(\pi(\mu))$. If $d(\pi(\lambda))_i > d(\pi(\mu))_i$ for some $i\leq k$, then $d(\mu)_i = d(\pi(\mu))_i$.
\begin{proof}
Write $\lambda = [x;(m,m+d(\lambda))]$. Then $\mu = [x;(m,m+d(\mu))]$, so
\begin{align*}
\pi(\lambda) &= [x;(m\wedge d(x), (m+d(\lambda))\wedge d(x))]\text{, and}\\
\pi(\mu) &= [x;(m\wedge d(x), (m+d(\mu))\wedge d(x))].
\end{align*}
Since $d(\lambda) \geq d(\mu),$ it follows that $\pi(\lambda) \in \Zz(\pi(\mu))$.

If $d(\pi(\lambda))_i > d(\pi(\mu))_i$, then $d(x)_i > m_i + d(\mu)_i$, so
\[
    d(\pi(\mu))_i = m_i + d(\mu)_i - m_i = d(\mu)_i.\qedhere
\]
\end{proof}
\end{lemma}

\begin{lemma}\label{faprooflem2}
Let $\Lambda$ be a $k$-graph and $\mu,\nu \in \wt\Lambda$. Then
\[
\pi(\MCE(\mu,\nu)) \subset \MCE(\pi(\mu),\pi(\nu)).
\]
\begin{proof}
Suppose that $\lambda \in \MCE(\mu,\nu)$. By Lemma~\ref{faprooflem1} we have $\pi(\lambda) \in \Zz(\pi(\mu)) \cap \Zz(\pi(\nu)),$ hence $d(\pi(\lambda)) \geq d(\pi(\mu)) \vee d(\pi(\nu))$.

It remains to prove that $d(\pi(\lambda)) = d(\pi(\mu)) \vee d(\pi(\nu))$. Suppose for a contradiction that there is some $i \leq k$ such that $d(\pi(\lambda))_i > \max\{d(\pi(\mu))_i, d(\pi(\nu))_i\}.$ By Lemma~\ref{faprooflem1} we then have $d(\pi(\mu))_i = d(\mu)_i$ and $d(\pi(\nu))_i = d(\nu)_i$. Then
$
    d(\lambda)_i \geq d(\pi(\lambda))_i > \max\{d(\mu)_i,d(\nu)_i\},
$
contradicting that $\lambda \in \MCE(\mu,\nu)$.
\end{proof}
\end{lemma}

\begin{lemma}\label{uniquepathextfromsmallgraph}
Let $\Lambda$ be a $k$-graph, and let $\mu,\lambda \in \iota(\Lambda^0)\wt\Lambda$ be such that $d(\lambda) = d(\mu)$ and $\pi(\lambda) = \pi(\mu)$. Then $\lambda = \mu$.
\begin{proof}
Since $\mu,\lambda \in \iota(\Lambda^0)\wt\Lambda$ and $d(\lambda) = d(\mu)$, we can write $\lambda = [x;(0,n)]$ and $\mu = [y;(0,n)]$ for some $x,y \in \partial \Lambda$ and $n \in \NN^k$. We will show that $(x;(0,n)) \sim (y;(0,n))$. Conditions \eqref{p2} and \eqref{p3} are trivially satisfied. Since
\[
    [x;(0,n \wedge d(x))] = \pi(\lambda) = \pi(\mu)= [y;(0, n \wedge d(y))],
\]
we have $(x;(0,n \wedge d(x))) \sim (y;(0, n \wedge d(y)))$. Hence $x(0 , n \wedge d(x)) = y(0 ,n \wedge d(y))$, and \eqref{p1} is satisfied.
\end{proof}
\end{lemma}

\begin{proof}[Proof of Theorem~\ref{finite aligndness preserved}]
The existence of $\wt\Lambda$ follows from Proposition~\ref{facprop}, and the embedding from  Proposition~\ref{iotamorph2}.

To check that $\wt\Lambda$ is finitely aligned, fix $\mu,\nu \in \wt\Lambda$, and $\alpha \in \iota(\Lambda^0)\wt\Lambda r(\mu)$. Then $|\MCE(\mu,\nu)| = |\MCE(\alpha\mu,\alpha\nu)|$. Since $\Lambda$ is finitely aligned, $|\MCE(\pi(\alpha\mu),\pi(\alpha\nu))|$ is finite. We will show that
$
    |\MCE(\alpha\mu,\alpha\nu)| = |\MCE(\pi(\alpha\mu),\pi(\alpha\nu))|.
$

It follows from Lemma~\ref{faprooflem2} that
$
    |\MCE(\alpha\mu,\alpha\nu)| \geq |\MCE(\pi(\alpha\mu),\pi(\alpha\nu))|.
$
For the opposite inequality, suppose $\lambda, \beta$ are distinct elements of $\MCE(\alpha\mu,\alpha\nu)$. Then $d(\lambda) = d(\beta)$. Since $r(\alpha\mu), r(\alpha\nu) \in \iota(\Lambda^0)$, Lemma~\ref{uniquepathextfromsmallgraph} implies that $\pi(\lambda) \neq \pi(\beta)$. So $|\MCE(\alpha\mu,\alpha\nu)| = |\MCE(\pi(\alpha\mu),\pi(\alpha\nu))|$.

For the last part of the statement, we prove the contrapositive. Suppose that $\wt \Lambda$ is not row-finite. Let $[x;m] \in \wt\Lambda^0$ and $i \leq k$ be such that $|[x;m]\wt\Lambda^{e_i}| = \infty$. Then for each $[y;(n,n+e_i)] \in [x;m]\wt\Lambda^{e_i}$ we have $[y;n] = [x;m]$, so $[x;(m,m+e_i)] \neq [y;(n,n+e_i)]$ only if \eqref{p1} fails. That is,
\begin{equation}\label{p1fail}
x(m \wedge d(x), (m+e_i)\wedge d(x)) \neq y(n \wedge d(y), (n+e_i)\wedge d(y)).
\end{equation}
Since $|[x;m]\wt\Lambda^{e_i}| = \infty$, there are infinitely many $[y;(n,n+e_i)] \in [x;m]\wt\Lambda^{e_i}$ satisfying \eqref{p1fail}. Hence $|x(m \wedge d(x))\Lambda^{e_i}| = \infty$.
\end{proof}

\begin{rmk}\label{crucial for FE set}
Suppose that $\Lambda$ is a finitely aligned $k$-graph, that $x \in \partial\Lambda$ and that $E \subset x(0)\Lambda$. Since $\iota:\Lambda \to \iota(\Lambda)$ is a bijective $k$-graph morphism, we have $E \in x(0)\FE(\Lambda)$ if and only if $\iota(E) \in [x;0]\FE(\iota(\Lambda))$.
\end{rmk}

The following results show how sets of minimal common extensions and finite exhaustive sets in a $k$-graph $\Lambda$ relate to those in $\wt\Lambda$.

\begin{prop}[{\cite[Lemma 2.25]{Farthing2008}}]\label{fe preserved}
Suppose that $\Lambda$ is a finitely aligned $k$-graph, and that $v \in \iota(\Lambda^0)$. Then $E \in v\FE(\iota(\Lambda))$ implies that $E \in v\FE(\wt\Lambda)$.
\end{prop}

\begin{lemma}\label{mce preserved}
Let $\Lambda$ be a finitely aligned $k$-graph and let $\mu,\nu \in \iota(\Lambda)$. Then $\MCE_{\iota(\Lambda)}(\mu,\nu) = \MCE_{\wt\Lambda}(\mu,\nu)$.
\begin{proof}
Since $\iota(\Lambda) \subset \wt\Lambda$, we have $\MCE_{\iota(\Lambda)}(\mu,\nu) \subset \MCE_{\wt\Lambda}(\mu,\nu)$. Suppose that $\lambda \in \MCE_{\wt\Lambda}(\mu,\nu)$. It suffices to show that $\lambda \in \iota(\Lambda)$. Write $\mu = [x;(0,n)], \nu = [y;(0,q)]$ and $\lambda = [z;(0,n \vee q)]$. Then $\lambda \in \Zz(\mu) \cap \Zz(\nu)$ implies that $d(z) \geq n \vee q$, hence $\lambda \in \iota(\Lambda)$.
\end{proof}
\end{lemma}

\begin{rmk}\label{lmin preserved}
Since there is a bijection from $\Lambda^{\min}(\mu,\nu)$ onto $\MCE(\mu,\nu)$, it follows from Lemma~\ref{mce preserved} that $\wt\Lambda^{\min}(\mu,\nu) = \iota(\Lambda)^{\min}(\mu,\nu)$ for all $\mu,\nu \in \iota(\Lambda)$.
\end{rmk}

\section{Topology of Path Spaces under Desourcification}\label{hrg path spaces under desource}

We extend the projection $\pi$ defined in \eqref{finite path hrg proj} to the set of infinite paths in $\wt\Lambda$, and prove that its restriction to $\iota(\Lambda^0)\wt\Lambda^\infty$ is a homeomorphism onto $\iota(\partial\Lambda)$. For $x \in \iota(\Lambda^0)\wt \Lambda^\infty$, let
$
    p_x = \bigvee \{ p \in \NN^k : x(0,p) \in \iota(\Lambda)\},
$
and define $\pi(x)$ to be the composition of $x$ with the inclusion of $\Omega_{k,p_x}$ in $\Omega_{k,d(x)}$. Then $\pi(x)$ is a $k$-graph morphism. Our goal for this section is the following theorem.

\begin{theorem}\label{hrg path spaces homeo}
Let $\Lambda$ be a row-finite $k$-graph. Then $\pi:\iota(\Lambda^0)\wt \Lambda^\infty \to \iota(\partial\Lambda)$ is a homeomorphism.
\end{theorem}

We first show that the range of $\pi$ is a subset of $\iota(\partial\Lambda)$.

\begin{prop}\label{pi into partial lambda}
Let $\Lambda$ be a finitely aligned $k$-graph. Let $x \in \iota(\Lambda^0)\wt{\Lambda}^\infty$. Suppose that $\{y_n : n \in \NN^k \} \subset \partial\Lambda$ satisfy
$[y_n;(0,n)]=x(0,n).$ Then
\begin{enumerate}
\renewcommand{\theenumi}{\roman{enumi}}
    \item $\ds \lim_{n \in \NN^k} \iota(y_n) = \pi(x)$ in $W_{\wt{\Lambda}}$; and
    \item there exists $y \in \partial \Lambda$ such that $\pi(x) = \iota(y)$, and for $m,n \in \NN^k$ with $m\leq n \leq p_x$ we have $\pi(x)(m,n) = \iota(y(m,n))$.
\end{enumerate}

\begin{proof} For part (i), fix a basic open set $\Zz(\mu \setminus G) \subset W_{\wt{\Lambda}}$ containing $\pi(x)$. Fix $n \geq N := \bigvee_{\nu \in G} d(\mu\nu)$. We first show that $\iota(y_n) \in \Zz(\mu)$. Since $\pi(x) \in \Zz(\mu)$, we have $\mu \in \iota(\Lambda)$. Since $n \geq d(\mu)$, we have $[y_n;(0,d(\mu))] = \mu.$

Let $\alpha = \iota^{-1}(\mu)$ and $z \in s(\alpha)\partial\Lambda$. Then
$
[y_n;(0,d(\mu))] = \mu = [\alpha z;(0,d(\mu))],
$
and \eqref{p1} gives
$
    \iota(y_n(0,d(\mu) \wedge d(y_n))) = \iota((\alpha z)(0,d(\mu))) = \iota(\alpha) = \mu.
$
So $\iota(y_n) \in \Zz(\mu)$.

We now show that $\iota(y_n) \notin \bigcup_{\nu \in G} \Zz(\mu\nu)$. Fix $\nu \in G$. If $d(y_n) \ngeq d(\mu\nu)$, then trivially we have $\iota(y_n) \notin \Zz(\mu\nu)$. Suppose that $d(y_n) \geq d(\mu\nu)$. Since $n \geq d(\mu\nu)$, we have
\[
x(0,d(\mu\nu)) = [y_n;(0,n)](0,d(\mu\nu)) = \iota(y_n)(0,d(\mu\nu)) \in \iota(\Lambda).
\]
So $\iota(y_n)(0,d(\mu\nu)) = x(0,d(\mu\nu)) = \pi(x)(0,d(\mu\nu)) \neq \mu\nu.$

For part (ii), recall that $\iota$ is injective, then we can define $y: \Omega_{k,p_x} \to \Lambda$ by $\iota(y(m,n)) = \pi(x)(m,n)$. So $\iota(y) = \pi(x)$. To see that $y \in \partial \Lambda$, fix $m \in \NN^k$ such that $m \leq d(y)$ and fix $E \in y(m)\FE(\Lambda)$. We seek $t \in \NN^k$ such that $y(m,
m+t) \in E$. Let $p := m+ \bigvee_{\mu \in E} d(\mu)$. Then since $m \leq d(y) = p_x$
\[
[y_p;(0,m)] = x(0,m) = \pi(x)(0,m) = \iota(y(0,m)) = [y(0,m)y';(0,m)]
\]
for some $y' \in y(m)\partial\Lambda$. So $(y_p;(0,m)) \sim (y(0,m)y';(0,m))$, hence
\[
    y_p(0,m \wedge d(y_p)) = (y(0,m)y')(0,m \wedge d(y(0,m)y')) = y(0,m)
\]
by \eqref{p1}. In particular, this implies that $y_p(m) = y(m)$. Since $y_p \in \partial\Lambda$, there exists $t \in \NN^k$ such that $y_p(m,m+t) \in E$. So $m+t \leq p$, and we have
\[
    \iota(y_p(m,m+t)) = [y_p;(0,p)](m,m+t) = x(0,p)(m,m+t) = x(m,m+t).
\]
So $x(m,m+t) \in \iota(\Lambda)$, giving
\[
    \iota(y_p(m,m+t)) = x(m,m+t) = \pi(x)(m,m+t) = \iota(y(m,m+t)).
\]
Finally, injectivity of $\iota$ gives $y(m,m+t) = y_p(m,m+t) \in E.$
\end{proof}
\end{prop}

The next few lemmas ensure that our definition of $\pi$ on $\wt\Lambda^\infty$ is compatible with \eqref{finite path hrg proj} when we regard finite paths as $k$-graph morphisms. The following lemma is also crucial in showing that $\pi$ is injective on $\iota(\Lambda^0)\wt\Lambda^\infty$.

\begin{lemma}\label{key to pi inj}
Let $\Lambda$ be a finitely aligned $k$-graph. Let $x \in \iota(\Lambda^0)\wt{\Lambda}^\infty$. Suppose that $w \in \partial\Lambda$ satisfies $\pi(x) = \iota(w)$. Then $x(0,n) = [w;(0,n)]$ for all $n \in \NN^k.$
\begin{proof}
Fix $n \in \NN^k$. Let $z \in \partial\Lambda$ be such that $x(0,n) = [z;(0,n)]$. We aim to
show that $(z;(0,n)) \sim (w;(0,n))$. That \eqref{p2} and \eqref{p3} hold follows immediately from their definitions. It remains to verify condition \eqref{p1}:
\begin{equation}\label{we need to show p3}
z(0,n\wedge d(z)) = w(0,n\wedge d(\omega)).
\end{equation}
Since $\pi(x) = \iota(w)$ we have $d(w) = p_x$. Thus
\[
    [w;(0,n\wedge p_x)] = \iota(w(0,n\wedge p_x)) = x(0, n \wedge p_x) = [z;(0,n \wedge p_x)].
\]
So $(w;(0,n\wedge p_x)) \sim (z;(0,n \wedge p_x))$. It then follows from \eqref{p1} that
\begin{equation}\label{nearly done}
w(0, n\wedge p_x) = z(0, n \wedge p_x).
\end{equation}
Hence $n \wedge d(z) \geq n \wedge p_x$. Furthermore,
\[
x(0,n\wedge d(z)) = [z;(0,n\wedge d(z))] = \iota(z(0, n \wedge d(z)))\in \iota(\Lambda)
\]
implies that $n \wedge p_x \geq n \wedge d(z)$. So $n \wedge
d(z) = n \wedge p_x$, and \eqref{nearly done} becomes \eqref{we need to
show p3}, as required.
\end{proof}
\end{lemma}
\begin{rmk}\label{iota of a segment}
Suppose that $\Lambda$ be a finitely aligned $k$-graph, and that $y \in \partial \Lambda$ and $m,n \in \NN^k$ satisfy $m \leq n \leq d(y)$. Then
\[
    [y;(m,n)] = [\sigma^m(y);(0,n-m)] = \iota(\sigma^m(y)(0,n-m)) = \iota(y(m,n)),
\]
So $[y;(m,n)] = \iota(y(m,n)).$
\end{rmk}

The next proposition shows that our definitions of $\pi$ for finite and infinite paths are compatible:

\begin{prop}\label{picompatible}
Let $\Lambda$ be a finitely aligned $k$-graph. Suppose that $x \in \wt\Lambda^\infty$, and $m \leq n\in\NN^k$. Then $\pi(x(m,n)) =
\pi(x)(m\wedge p_x, n\wedge p_x)$.
\begin{proof}
Fix $y \in \partial\Lambda$ such that $\pi(x) = \iota(y)$. Then
\begin{align*}
\pi(x(m,n)) &= \pi([y;(m,n)] \qquad\text{by Lemma~\ref{key to pi inj}}\\
&= [y;(m\wedge p_x, n \wedge p_x)] \qquad\text{since $d(y) = p_x$}\\
&= \iota(y(m\wedge p_x, n \wedge p_x)) \qquad\text{by Remark~\ref{iota of a segment}}\\
&= \pi(x)(m\wedge p_x, n \wedge p_x) \qquad\text{by Proposition~\ref{pi into partial lambda}(ii) }. \qedhere
\end{align*}
\end{proof}
\end{prop}

We can now show that $\pi$ restricts to a homeomorphism of $\iota(\Lambda^0)\wt\Lambda^\infty$ onto $\iota(\partial\Lambda)$. We first show that it is a bijection, then show it is continuous. Openness is the trickiest part, and the proof of it completes this section.

\begin{prop}\label{pi is a bij}
Let $\Lambda$ be a finitely aligned $k$-graph. Then the map $\pi:\iota(\Lambda^0)\wt \Lambda^\infty \to \iota(\partial\Lambda)$ is a bijection.
\begin{proof}
That $\pi$ is injective follows from Lemma~\ref{key to pi inj}. To see that $\pi$ is onto $\iota(\partial\Lambda)$, let $w \in \partial\Lambda$ and define $x:\Omega_k
\to \wt{\Lambda}$ by $x(p,q) = [w;(p,q)]$. Then $p_x = d(w)$, and $r(x) \in \iota(\Lambda)$. To see that $\pi(x) = \iota(w)$, fix $m,n \in \NN^k$ with $m \leq n \leq d(w)$. Then
\begin{align*}
\pi(x)(m,n) &= x(m,n) \qquad\text{by Proposition~\ref{picompatible}}\\
&= [w;(m,n)] \qquad\text{by Lemma~\ref{key to pi inj}}\\
&= \iota(w(m,n)) \qquad\text{by Remark~\ref{iota of a segment}}\\
&= \iota(w)(m,n) \qquad\text{by Proposition~\ref{extend kgm to nonfinite paths}.}
\qedhere
\end{align*}
\end{proof}
\end{prop}

\begin{prop}\label{pi is cont}
Let $\Lambda$ be a finitely aligned $k$-graph. Then $\pi:\iota(\Lambda^0)\wt \Lambda^\infty \to \iota(\partial\Lambda)$ is continuous.
\begin{proof}
Fix a basic open set $\Zz(\mu \setminus G) \subset W_{\wt\Lambda}$. If $\Zz(\mu \setminus
G) \cap \iota(\partial \Lambda) = \emptyset$, then $\pi^{-1}(\Zz(\mu \setminus
G) \cap \iota(\partial \Lambda)) = \emptyset$ is open. Suppose that $\Zz(\mu \setminus
G) \cap \iota(\partial \Lambda) \neq \emptyset$, and fix $y \in \Zz(\mu \setminus
G) \cap \iota(\partial \Lambda)$. Let $F = G \cap \iota(\Lambda)$. We will show that
\begin{equation}\label{req for pi cts}
    \pi^{-1}(y) \in \Zz(\mu \setminus F) \cap \big( \wt\Lambda^\infty \cap r^{-1}(\iota(\Lambda))\big)
    \subset \pi^{-1}(\Zz(\mu \setminus G) \cap \iota(\partial \Lambda)).
\end{equation}
Since $y \in \Zz(\mu)$, it follows that $\pi^{-1}(y) \in \Zz(\mu)$. To see that $\pi^{-1}(y) \notin \bigcup_{\beta \in F} \Zz(\mu\beta)$, fix $\beta \in F$. First suppose that $d(\mu\beta) \nleq d(y)$. Then $\pi^{-1}(y)(0,d(\mu\beta)) \notin \iota(\Lambda)$. Since $\mu\beta \in \iota(\Lambda)$, we have $\pi^{-1}(y)(0,d(\mu\beta)) \neq \mu\beta$. Now suppose that $d(\mu\beta) \leq d(y)$, then
\[
    \pi^{-1}(y)(0, d(\mu\beta)) = y(0,d(\mu\beta)) \neq \mu\beta.
\]

We now show that
$
    \Zz(\mu \setminus F) \cap \iota(\Lambda^0)\wt\Lambda^\infty \subset \pi^{-1}(\Zz(\mu \setminus G) \cap \iota(\partial \Lambda)).
$
Let $z \in \Zz(\mu \setminus F)\cap\iota(\Lambda^0)\wt\Lambda^\infty$. It suffices to show that $\pi(z) \in \Zz(\mu \setminus G)$. Firstly, $\pi(z)(0,d(\mu)) = z(0,d(\mu)) = \mu \in \iota(\Lambda)$. To see that $\pi(z) \notin \bigcup_{\nu \in G}\Zz(\mu\nu)$, fix $\nu \in G$. If $d(\mu\nu) \nleq d(\pi(z))$, then trivially $\pi(z) \notin \Zz(\mu\nu)$. Suppose that $d(\mu\nu) \leq d(\pi(z))$. If $\nu \notin \iota(\Lambda)$, then $\pi(z)(0,d(\mu\nu)) \neq \mu\nu$. Otherwise, $\nu \in \iota(\Lambda)$, then $\nu \in F$ and we have
$
    \pi(z)(0,d(\mu\nu)) = z(0,d(\mu\nu)) \neq \mu\nu. %\qedhere
$
\end{proof}
\end{prop}

\begin{prop}\label{pi is open}
Let $\Lambda$ be a row-finite $k$-graph. Then $\pi:\iota(\Lambda^0)\wt\Lambda^\infty \to
\iota(\partial\Lambda)$ is open.
\begin{proof}
Fix $\pi(y) \in \pi(\Zz(\mu \setminus G) \cap \iota(\Lambda^0)\wt\Lambda^\infty)$. Let $\omega \in \partial \Lambda$ be such that $\pi(y) = \iota(\omega)$. Define $\lambda := y(0,\bigvee_{\nu\in G} d(\mu\nu))$, and
\[
    F := \bigcup\{ s(\pi(\lambda))\iota(\Lambda^{e_i}) :d(\lambda)_i> d(\pi(y))_i\}.
\]
We claim that
\[
    \pi(y) \in \Zz(\pi(\lambda) \setminus F) \cap \iota(\partial\Lambda) \subset \pi(\Zz(\mu \setminus G) \cap \iota(\Lambda^0)\wt\Lambda^\infty).
\]

First we show that $\pi(y) \in \Zz(\pi(\lambda))$. It follows from Lemma~\ref{key to pi inj} that $\pi(\lambda) = [\omega;(0,d(\lambda) \wedge d(\omega))]$. Since $d(\omega) = d(\pi(y))$, Proposition~\ref{picompatible} implies that
\[
\pi(y)(0, d(\pi(\lambda))) = \pi(y)(0,d(\lambda) \wedge d(\omega)) = \pi(y(0,d(\lambda))) = \pi(\lambda).
\]
Now we show that $\pi(y) \notin \bigcup_{f \in F}\Zz(\pi(\lambda) f)$. Fix $f \in F$; say $d(f) = e_i$. Then by definition of $F$, $d(\lambda)_i > d(\pi(y))_i = d(\omega)_i$, and thus
\[
d(\pi(\lambda))_i = \min\{d(\lambda)_i,d(\omega)_i\} = d(\omega)_i = d(\pi(y))_i.
\]
So $d(\pi(y)) \ngeq d(\pi(\lambda)f)$, and hence $\pi(y) \notin \Zz(\pi(\lambda)f)$ as required.

Now we show that $\Zz(\pi(\lambda) \setminus F) \cap \iota(\partial\Lambda) \subset \pi(\Zz(\mu \setminus G) \cap \iota(\Lambda^0)\wt\Lambda^\infty).$ Let $\pi(\beta) \in \Zz(\pi(\lambda) \setminus F) \cap \iota(\partial\Lambda)$. We aim to show that $\beta \in \Zz(\mu \setminus G).$ Since $\Zz(\lambda) \subset \Zz(\mu \setminus G)$, it suffices to show that $\beta \in \Zz(\lambda)$. Clearly $\beta \in \Zz(\pi(\lambda) \setminus F)$. If $d(\lambda) = d(\pi(\lambda))$ then $\pi(\lambda) = \lambda$ and we are done. Suppose that $d(\lambda) > d(\pi(\lambda))$, and let $\tau = \beta(d(\pi(\lambda)), d(\lambda)).$ We know that $\beta \in \Zz(\pi(\lambda))$. We
aim to use Lemma~\ref{uniquepathextfromsmallgraph} to show that $\tau = \lambda(d(\pi(\lambda)),
d(\lambda)).$ Fix $i \leq k$ such that $d(\lambda)_i > d(\pi(\lambda))_i$. Then since $d(\pi(\lambda)) = d(\lambda) \wedge d(\omega)$, we have $d(\lambda)_i > d(\omega)_i = d(\pi(y))_i$. Now $\beta \in \Zz(\pi(\lambda) \setminus F)$ implies that $\tau(0,e_i) \notin F$. In particular, $\tau(0,e_i) \notin \iota(\Lambda).$ We claim that $d(\pi(\tau)) = 0$. Suppose, for a contradiction, that $d(\pi(\tau))_j > 0$ for some $j \leq k$. Then $\pi(\tau)(0,e_j) = \tau(0,e_j) \notin \iota(\Lambda).$ But $\pi(\tau) \in \iota(\Lambda)$ by definition of $\pi$. So we must have $d(\pi(\tau)) = 0$, which implies that
\[
    \pi(\tau) = r(\tau) = s(\pi(\lambda)) = \pi(\lambda(d(\pi(\lambda)), d(\lambda))).
\]
Now Lemma~\ref{uniquepathextfromsmallgraph} implies that $\tau = \lambda(d(\pi(\lambda)),
d(\lambda)).$ Then
\[
    \beta(0,\lambda) = \beta(0,d(\pi(\lambda)))\tau = \pi(\lambda)\lambda(d(\pi(\lambda)),d(\lambda)) = \lambda. \qedhere
\]
\end{proof}
\end{prop}

\begin{example}
We can see that $\pi$ is not open for non-row-finite graphs by considering the $1$-graph $E$ from Example~\ref{desource ex} with `desourcification' $\wt E$. Observe that $\Zz(\mu_1)\cap\iota(E^0)\wt\Lambda^\infty = \{\mu_1\mu_2\cdots\}$ is open in $\wt E$, and $\pi(\Zz(\mu_1)\cap\iota(E^0)\wt E^\infty) = \{v\}$. Since $\partial E = E$, any basic open set in $\partial E$ containing $v$ is of the form $\Zz(v \setminus G)$ for some finite $G \subset E^1$. Since $E^1$ is infinite, there is no finite $G \subset E^1$ such that $\Zz(v \setminus G) \subset \{v\}$. Hence $\{v\}$ is not open in $E$, and $\pi$ is not an open map.
\end{example}

\begin{proof}[Proof of Theorem~\ref{hrg path spaces homeo}]
  Propositions~\ref{pi is a bij},~\ref{pi is cont} and~\ref{pi is open} say precisely that $\pi$ is a bijection, is continuous, and is open.
\end{proof}%

\begin{rmk}
Although $\pi|_{\iota(\Lambda^0)\wt\Lambda^\infty}$ is open for all row-finite $k$-graphs, it behaves particularly well with respect to cylinder sets for locally convex $k$-graphs. The following discussion and example arose in preliminary work on a proof that $\pi$ is open when $\Lambda$ is row-finite and locally convex. We have retained this example since it helps illustrate some of the issues surrounding the map $\pi$.

Denote our standard topology for a finitely $k$-graph by $\tau_1$. The
collection $\{\Zz(\mu) : \mu \in \Lambda\}$ of cylinder sets also form a base for a topology: they cover $W_\Lambda$, and if $x \in \Zz(\lambda) \cap \Zz(\nu)$, then $x \in \Zz(x(0,d(\lambda)\vee d(\nu))) \subset \Zz(\lambda) \cap \Zz(\nu)$. This topology, denoted $\tau_2$, is not necessarily Hausdorff: we cannot separate any edge from its range: if $r(f) \in \Zz(\mu)$ then $\mu = r(f)$, and thus $f \in \Zz(\mu)$.

It may seem reasonable to expect that $\{\Zz(\mu) \cap \partial \Lambda : \mu \in \Lambda\}$ is a base for the restriction of $\tau_1$ to $\partial\Lambda$. However, this is not so. To see why, consider the $2$-graph of Example~\ref{boundary different to leq infty}. Let $y$ be the boundary path beginning with $f_0$. So $x,y \in
\partial\Lambda$. Let $\mu$ be such that $x \in \Zz(\mu)$. Then $\mu = x_0 \dots x_n$ for some $n \in
\NN$, so $y \in \Zz(\mu)$ also. So the topology $\tau_1$ is not Hausdorff even when restricted to $\partial \Lambda$. Endowed with $\tau_2$, it is easy to see how to separate these two points: $y \in \Zz(f_0) \cap \partial \Lambda$ and $x \in \Zz(r(x)\setminus \{f_0\})\cap \partial \Lambda$, and these two sets are disjoint.

If we restrict ourselves to locally convex $k$-graphs, $\tau_1$ and $\tau_2$ do restrict to the same topology on $\partial \Lambda$: certainly, for each $\mu \in \Lambda$, we can realise a cylinder set $\Zz(\mu)$ as a set of the form $\Zz(\mu \setminus G)$ by taking $G = \emptyset$. Now suppose that $x \in \Zz(\mu \setminus G) \cap \partial \Lambda$. We claim that with
\[
    \nu_x := x(0,\big( \bigvee_{\alpha \in G} d(\mu\alpha) \big) \wedge d(x)),
\]
we have
$
x \in \Zz(\nu_x) \cap \partial\Lambda \subset \Zz(\mu \setminus G) \cap \partial\Lambda.
$
Clearly we have $x \in \Zz(\nu_x) \cap \partial\Lambda$. The containment requires a little more work. Clearly $y \in \Zz(\mu)$. Fix $\alpha \in G$. We will show that $y \notin \Zz(\mu\alpha)$. If $d(y) \ngeq d(\mu\alpha)$, then trivially $y \notin \Zz(\mu\alpha)$. Suppose that $d(y) \geq d(\mu\alpha)$. We claim that $d(x) \geq d(\mu\alpha)$: suppose, for a contradiction, that $d(x) \ngeq d(\mu\alpha)$. Then there exists $i \leq k$ such that $d(x)_i < d(\mu\alpha)_i$. Then $d(x)_i = d(\nu_x)_i$. Since $x \in \partial\Lambda$, we must have $x(d(\nu_x))\Lambda^{e_i} \notin x(d(\nu_x))\FE(\Lambda).$ Since $\Lambda$ is locally convex, Lemma~\ref{v lambda ei non empty then exh} implies that
$
y(d(\nu_x))\Lambda^{e_i} = x(d(\nu_x))\Lambda^{e_i} = \emptyset.
$
So $d(y)_i = d(\nu_x)_i = d(x)_i < d(\mu\alpha)_i$, a contradiction. Hence $d(x) \geq d(\mu\alpha)$. This implies that $d(\nu_x) \geq d(\mu\alpha)$. So \[y(0,d(\mu\alpha)) = v_x(0,d(\mu\alpha)) = x(0,d(\mu\alpha)) \neq \mu\alpha.\]
\end{rmk}

\begin{prop}\label{local convexity pi preserves basic open sets}
Suppose that $\Lambda$ is a row-finite, locally convex $k$-graph, and let $\mu \in \iota(\Lambda^0)\wt\Lambda$. Then
$
    \pi(\Zz(\mu) \cap \iota(\Lambda^0)\wt\Lambda^\infty) = \Zz(\pi(\mu))\cap\iota(\partial\Lambda).
$
In particular, $\pi$ is open.
\begin{proof}
We first show that $\pi(\Zz(\mu) \cap \iota(\Lambda^0)\wt\Lambda^\infty) \subset \Zz(\pi(\mu))\cap\iota(\partial\Lambda)$. Suppose that $\pi(y) \in \pi(\Zz(\mu \setminus G) \cap \iota(\Lambda^0)\wt\Lambda^\infty)$. Trivially $\pi(y) \in \iota(\partial \Lambda)$. We will show that $\pi(y) \in\Zz(\pi(\mu) \setminus \pi(G))$. Since $y(0, d(\mu)) = \mu$, we have
\[
    \pi(\mu) = \pi(y(0,d(\mu))) = \pi(y)(0,d(\mu)\wedge d(\pi(y))).
\]
 So $\pi(y) \in \Zz(\pi(\mu))$. Furthermore, $d(\pi(\mu)) = d(\mu) \wedge d(\pi(y))$.

Fix $\nu \in G$. We will show that $\pi(y) \notin \Zz(\pi(\mu\nu))$. Since $y \in \Zz(\mu \setminus G)$, we have $y(0,d(\mu\nu)) \neq \mu\nu$. Since $d(y(0,d(\mu\nu))) = d(\mu\nu)$ and $r(y)=r(\mu\nu) \in \iota(\Lambda^0)$, Lemma
~\ref{uniquepathextfromsmallgraph} implies that
\[
    \pi(\mu\nu) \neq \pi(y(0,d(\mu\nu))) = \pi(y)(0,d(\mu\nu)\wedge d(\pi(y))).
\]
So $\pi(\Zz(\mu \setminus G) \cap \iota(\Lambda^0)\wt\Lambda^\infty) \subset \Zz(\pi(\mu) \setminus \pi(G)) \cap \iota(\partial \Lambda).$

Now suppose that $\iota(\omega) \in\Zz(\pi(\mu))\cap\iota(\partial\Lambda)$, and let $y = \pi^{-1}(\iota(\omega)).$ We show that $y \in \Zz(\mu)$. Write $\mu = [z;(0,d(\mu))]$. Then $\pi(\mu) = [z;(0,d(\mu)\wedge d(z))]$ and $y(0,d(\mu)) = [\omega;(0,d(\mu))]$. We claim that $(z;(0,d(\mu))) \sim (\omega;(0,d(\mu)))$. That \eqref{p2} and \eqref{p3} hold follows immediately from their definition. To show that \eqref{p1} is satisfied, we must show that $z(0,d(\mu) \wedge d(z)) = w(0,d(\mu)\wedge d(w))$. Since $\pi(y) \in \Zz(\pi(\mu))$, we have $y \in \Zz(\pi(\mu))$. Then
\[
    [\omega;(0,d(\pi(\mu)))] = y(0,d(\pi(\mu))) = \pi(\mu) = [z;(0,d(\mu)\wedge d(z))].
\]
So $(\omega;(0,d(\pi(\mu)))) \sim (z;(0,d(\mu)\wedge d(z)))$. Then \eqref{p1} implies that
\[
    \omega(0,d(\pi(\mu))) = \omega(0,d(\pi(\mu)) \wedge d(\omega)) = z(0,d(\mu)\wedge d(z)),
\]
and $d(\pi(\mu)) = d(\mu) \wedge d(z)$. We will show $d(\mu) \wedge d(w) = d(\pi(\mu))$. Fix $i \leq k$. We argue the following cases separately:
\begin{enumerate}
\item If $d(\pi(\mu))_i = d(\mu)_i$, we have $d(w) \geq d(\pi(\mu)) = d(\mu)_i$. Hence $(d(\mu) \wedge d(w))_i = d(\mu)_i = d(\pi(\mu))_i$.
\item If $d(\pi(\mu))_i < d(\mu)_i$, it requires a little more work:
\end{enumerate}
Since $d(\mu)_i > d(\pi(\mu))_i = \min\{d(\mu)_i,d(z)_i\}$, we have $d(\pi(\mu))_i = d(z)_i$. Then $z \in \partial\Lambda$ implies that $z(d(\pi(\mu)))\Lambda^{e_i} \notin z(d(\pi(\mu)))\FE(\Lambda)$. By Lemma~\ref{v lambda ei non empty then exh}, we have $z(d(\pi(\mu)))\Lambda^{e_i} = \emptyset$, and hence $\omega(d(\pi(\mu)))\Lambda^{e_i} = \emptyset$. So $d(\omega)_i = d(\pi(\mu))_i < d(\mu)_i$, giving $(d(\mu) \wedge d(\omega))_i = d(\omega)_i = d(\pi(\mu))_i$.
\end{proof}
\end{prop}

\section{High-Rank Graph $C^*$-algebras}\label{hrg alg}

\begin{defn}
Let $\Lambda$ be a finitely aligned $k$-graph. A \emph{Cuntz-Krieger $\Lambda$-family} in a $C^*$-algebra $B$ is a collection $\{t_\lambda : \lambda \in \Lambda\}$ of partial isometries satisfying
\renewcommand{\theenumi}{CK\arabic{enumi}}
\begin{enumerate}
\item $\{s_v : v \in \Lambda^0\}$ is a set of mutually orthogonal projections;\label{ck1}
\item $s_\mu s_\nu = s_{\mu\nu}$ whenever $s(\mu) = r(\nu)$;\label{ck2}
\item $s_\mu^*s_\nu = \sum_{(\alpha,\beta) \in \Lambda^{\min}(\mu,\nu)} s_\alpha s_\beta^*$ for all $\mu,\nu \in \Lambda$; and\label{ck3}
\item $\prod_{\mu\in E} (s_v - s_\mu s_\mu^*) = 0$ for every $v \in \Lambda^0$ and $E \in v\FE(\Lambda)$.\label{ck4}
\end{enumerate}
\end{defn}

The $C^*$-algebra $C^*(\Lambda)$ of a $k$-graph $\Lambda$ is the universal $C^*$-algebra generated by a Cuntz-Krieger $\Lambda$-family $\{s_\lambda: \lambda \in \Lambda \}$.

\begin{rmk}\label{cindys error}
 The following Theorem appears as \cite[Theorem 2.28]{Farthing2008}. Farthing alerted us to an issue in the proof of the theorem. It contains a claim which is proved in cases, and in the proof of Case 1 of the claim (on page 189), there is an error when $i_0$ is such that $m_{i_0} = d(x)_{i_0}+1$. Then $a_{i_0} = d(x)_{i_0}$, and \cite[Equation (2.13)]{Farthing2008} gives $t_{i_0} \leq d(z)_{i_0}$; not $t_{i_0} \geq d(z)_{i_0}$ as stated.
\end{rmk}

\begin{theorem}\label{c*lambda morita to c*wtlambda}
Let $\Lambda$ be a row-finite $k$-graph. Let $C^*(\Lambda)$ and $C^*(\wt\Lambda)$ be generated by the Cuntz-Krieger families $\{s_\lambda : \lambda \in \Lambda\}$ and $\{t_\lambda : \lambda \in \wt\Lambda\}$. Then the sum $\sum_{v \in \iota(\Lambda^0)}t_v$ converges strictly to a full projection $p \in M(C^*(\wt\Lambda))$ such that $p C^*(\wt\Lambda)p = C^*(\{t_{\iota(\lambda)} : \lambda \in \Lambda\})$, and $s_\lambda \mapsto t_{\iota(\lambda)}$ determines an isomorphism $\varsigma:C^*(\Lambda) \cong pC^*(\wt\Lambda)p$.
\end{theorem}

Before proving Theorem~\ref{c*lambda morita to c*wtlambda}, we need the following results.

\begin{prop}[{\cite[Theorem 2.26]{Farthing2008}}]\label{restr is a ckfam}
Let $\Lambda$ be a finitely aligned $k$-graph. If $\{t_\lambda:\lambda \in \wt\Lambda\}$ is a Cuntz-Krieger $\wt\Lambda$-family, then $\{t_\lambda: \lambda \in \iota(\Lambda)\}$ is a Cuntz-Krieger $\iota(\Lambda)$-family.
\end{prop}

\begin{rmk}\label{iota preserves cstar alg}
Let $\Lambda$ be a finitely aligned $k$-graph. It follows from the universal properties of $C^*(\Lambda)$ and $C^*(\iota(\Lambda))$ that $C^*(\Lambda) \cong C^*(\iota(\Lambda))$.
\end{rmk}

\begin{prop}[{\cite[Theorem 2.27]{Farthing2008}}]\label{subalg iso to cstarlambda}
Let $\Lambda$ be a finitely aligned $k$-graph, and let $\{t_\lambda : \lambda \in \wt\Lambda\}$ be the universal Cuntz-Krieger $\wt\Lambda$-family which generates $C^*(\wt\Lambda)$. Then $C^*(\Lambda)$ is isomorphic to the subalgebra of $C^*(\wt\Lambda)$ generated by $\{t_\lambda : \lambda \in \iota(\Lambda)\}$.
\end{prop}

\begin{lemma}\label{boundary path choice for paths off iota lambda}
Suppose that $\Lambda$ is a finitely aligned $k$-graph. Let $\lambda \in \wt\Lambda$, and let $\lambda' = \lambda(d(\pi(\lambda)), d(\lambda))$. Suppose that $x \in \partial\Lambda$ satisfies $\iota(r(x)) = r(\lambda')$ and $d(x) \wedge d(\lambda') = 0$. Then $\lambda' = [x; (0, d(\lambda'))].$
\begin{proof}

Write $\lambda = [y;(0,d(\lambda))]$, then $\lambda' = [y;(d(\lambda) \wedge d(y), d(\lambda))].$ We must show that $(y;(d(\lambda)\wedge d(y),d(\lambda)) \sim (x;(0,d(\lambda'))$. That conditions \eqref{p2} and \eqref{p3} hold follows immediately from their definitions. It remains to show that \eqref{p1} is satisfied. Since  $d(x) \wedge d(\lambda') = 0$, it suffices to show that $y(d(\lambda) \wedge d(y)) = x(0).$
We have
\[
   \iota(x(0)) = \iota(r(x)) = r(\lambda') = [y;d(\lambda) \wedge d(y)] = \iota(y(d(\lambda) \wedge d(y))).
\]
Injectivity of $\iota$ then gives $y(d(\lambda) \wedge d(y)) = x(0)$.
\end{proof}
\end{lemma}

\begin{lemma}\label{Glambda is FE}
Let $\lambda \in \wt\Lambda$. Let $\lambda' = \lambda(d(\pi(\lambda)), d(\lambda))$ and define
\[
    G_\lambda := \bigcup_{i=1}^k \{\alpha \in s(\pi(\lambda))\iota(\Lambda)^{e_i} : \MCE(\alpha,\lambda') = \emptyset \}.
\]
Then $G_\lambda \cup \{\lambda'\} \in s(\pi(\lambda))\FE(\wt\Lambda)$.
\begin{proof}
Fix $\mu \in s(\pi(\lambda))\wt\Lambda$, and suppose that $\MCE(\mu,\alpha) = \emptyset$ for all $\alpha \in G_\lambda$. We will show that $\MCE(\mu,\lambda') \neq \emptyset$. Fix $\nu \in s(\mu)\wt\Lambda^{d(\mu)\vee d(\lambda') - d(\mu)}$. Then $d(\mu\nu) \geq d(\lambda')$. It suffices to show that $\MCE(\mu\nu, \lambda') \neq \emptyset$. Write $\mu\nu = [z ; (0, d(\mu\nu))]$.

We first show that $d(\lambda') \wedge d(\pi(\mu\nu)) = 0$. Suppose for a contradiction that $d(\lambda') \wedge d(\pi(\mu\nu)) >0$. So we have $d(\lambda') \wedge d(\mu\nu) \wedge d(z) > 0$, hence there exists $i \leq k$ such that $d(\lambda')_i, d(\mu\nu)_i$, and $d(z)_i$ are all greater than zero. Then
$
(\mu\nu)(0,e_i) = [z; (0,e_i)] = \iota(z)(0,e_i) \in \iota(\Lambda).
$
Since $\pi|_{\iota(\Lambda)} = \id_{\iota(\Lambda)}$ and $\pi(\lambda') = s(\pi(\lambda)) \neq \lambda'$, we have $\lambda' \notin \iota(\Lambda)$. This implies that $(\mu\nu)(0,e_i) \neq \lambda'(0,e_i)$. So $\MCE((\mu\nu)(0,e_i), \lambda') = \emptyset$, and thus $(\mu\nu)(0,e_i) \in G_\lambda$. But $\MCE(\mu\nu(0,e_i),\mu\nu) \neq \emptyset$, which implies that $\MCE(\mu,\mu\nu(0,e_i)) \neq \emptyset$. This contradicts our supposition that  $\MCE(\mu,\alpha) = \emptyset$ for all $\alpha \in G_\lambda$. So $d(\lambda') \wedge d(\pi(\mu\nu)) = 0$.

Since $d(\mu\nu) \geq d(\lambda')$, we have
\[
d(z) \wedge d(\lambda') = d(z) \wedge d(\mu\nu) \wedge d(\lambda') = d(\pi(\mu\nu)) \wedge d(\lambda') = 0
\]

Since $r(\lambda') = r(\mu\nu) = \iota(r(z))$, it follows from Lemma~\ref{boundary path choice for paths off iota lambda} that $\lambda' = [z;(0,\lambda')]$. Thus $\mu\nu = [z ; (0,\mu\nu)] \in \MCE(\mu\nu, \lambda')$.
\end{proof}
\end{lemma}

\begin{proof}[Proof of Theorem~\ref{c*lambda morita to c*wtlambda}]
Let $A:= C^*(\{t_\lambda: \lambda \in \iota(\Lambda)\})$. Then $A \cong C^*(\Lambda)$ by Proposition~\ref{subalg iso to cstarlambda}. We will show that $A$ is a full corner of $C^*(\wt\Lambda)$.

Following the argument of \cite[Lemma 2.10]{Raeburn2005}, the sum $\sum_{v \in \iota(\Lambda^0)} t_v$ converges strictly in $M(C^*(\wt\Lambda))$ to a projection $p$ satisfying
\begin{equation}\label{proj in multi}
p t_\lambda t_\mu^* p = \begin{cases} t_\lambda t_\mu^* & \text{if } \wt r(\lambda), \wt r(\mu) \in \iota(\Lambda^0); \\ 0 & \text{otherwise.}\end{cases}
\end{equation}
The standard argument shows that $p$ is a full projection in $M(C^*(\wt\Lambda))$. It follows from \eqref{proj in multi} that $A \subset p C^*(\wt\Lambda) p$. Now suppose that $\lambda,\mu \in \iota(\Lambda^0)\wt\Lambda$. We will show that $pt_\lambda t_\mu^* p \in A$. If $\wt s(\lambda) \neq \wt s(\mu)$, then \eqref{ck1} implies that $pt_\lambda t_\mu^* p = 0 \in A$. Suppose that $\wt s(\lambda) = \wt s(\mu)$. We first show that \begin{equation}\label{eqn in hrg corner thm}
\lambda(d(\pi(\lambda)), d(\lambda)) = \mu(d(\pi(\mu)), d(\mu)).
\end{equation}
Let $x,y \in \partial\Lambda$ such that $\lambda = [x;(0,d(\lambda))]$ and $\mu = [y; (0,d(\mu))]$. Let
\begin{align*}
\lambda' &= \lambda(d(\pi(\lambda)), d(\lambda)) = [x;(d(\lambda) \wedge d(x), d(\lambda)] \qquad{\text{ and}}\\
\mu' &= \mu(d(\pi(\mu)), d(\mu)) = [y;(d(\mu) \wedge d(y), d(\mu))].
\end{align*}
We claim that $\lambda' = \mu'$. Condition~\eqref{p2} is trivially satisfied, and \eqref{p1} and \eqref{p3} follow from the vertex equivalence $[x;d(\lambda)] = \wt s(\lambda) = \wt s(\mu) = [y;d(\mu)]$. Hence $\lambda'= \mu'$.

\setcounter{theorem}{3}
\begin{claim}\label{claim for corner thm}
Let $G_\lambda := \bigcup_{i=1}^k \{\alpha \in s(\pi(\lambda))\iota(\Lambda)^{e_i} : \MCE(\alpha,\lambda') = \emptyset \}.$ Then
\[t_{\lambda'}t_{\lambda'}^* = \prod_{\alpha \in G_\lambda} \big(t_{s(\pi(\lambda))} - t_\alpha t_\alpha^*\big)\]
\begin{proof}
Lemma~\ref{Glambda is FE} implies that $G_\lambda \cup \{\lambda'\}$ is finite exhaustive, so \eqref{ck4} implies that
\[
    \prod_{\beta \in G_\lambda \cup \{\lambda'\}} \big( t_{s(\pi(\lambda))} - t_\beta t_\beta^* \big) =0.
\]

Furthermore,
\begin{align*}
\prod_{\beta \in G_\lambda \cup \{\lambda'\}} &\big( t_{s(\pi(\lambda))} - t_\beta t_\beta^* \big) = \Big( \prod_{\alpha \in G_\lambda} \big( t_{s(\pi(\lambda))} - t_\alpha t_\alpha^* \big) \Big) (t_{s(\pi(\lambda))} - t_{\lambda'} t_{\lambda'}^*) \\
&= \Big(\prod_{\alpha \in G_\lambda} ( t_{s(\pi(\lambda))} - t_\alpha t_\alpha^* )\Big) - \Big(t_{\lambda'} t_{\lambda'}^* \prod_{\alpha \in G_\lambda} ( t_{s(\pi(\lambda))} - t_\alpha t_\alpha^* )\Big).
\end{align*}

Fix $\alpha \in G_\lambda$. By \cite[Lemma 2.7(i)]{RSY2004},
\[
t_{\lambda'} t_{\lambda'}^* ( t_{s(\pi(\lambda))} - t_\alpha t_\alpha^*) = t_{\lambda'} t_{\lambda'}^* - \sum_{\gamma \in \MCE(\lambda',\alpha)} t_\gamma t_\gamma^* = t_{\lambda'} t_{\lambda'}^*.
\]
So
\[
0 = \prod_{\beta \in G_\lambda \cup \{\lambda'\}} \big( t_{s(\pi(\lambda))} - t_\beta t_\beta^* \big) = \prod_{\alpha \in G_\lambda} \big( t_{s(\pi(\lambda))} - t_\alpha t_\alpha^* \big) - t_{\lambda'} t_{\lambda'}^*. \pc\qedhere
\]
\end{proof}
\end{claim}

Now we put the pieces together:
\begin{align*}
p t_\lambda t_\mu^* p &= t_\lambda t_\mu^*\\
&= t_{\pi(\lambda)} t_{\lambda'} t_{\lambda'}^* t_{\pi(\mu)}^* \qquad\text{by \eqref{eqn in hrg corner thm}}\\
&= t_{\pi(\lambda)} \prod_{\alpha \in G_\lambda} \big(t_{s(\pi(\lambda))} - t_\alpha t_\alpha^*\big) t_{\pi(\mu)}^* \qquad\text{by Claim~\ref{claim for corner thm}.}
\end{align*}
which is an element of $A$ since $\pi(\lambda),\pi(\mu),\alpha \in \iota(\Lambda)$ for all $\alpha \in G_\lambda$. So $A = p C^*(\wt\Lambda) p$.
%
%To show that $A$ is full in $C^*(\wt\Lambda)$, suppose that $J$ is an ideal in $C^*(\wt\Lambda)$ such that $A \subset J$. Let $v \in \wt\Lambda^0$, and $\alpha \in \pi(v)\wt\Lambda v$. Then $t_\alpha = t_{r(\alpha)}t_\alpha$. Since $r(\alpha) = \pi(v) \in \iota(\Lambda)$, we have $t_{r(\alpha)} \in A \subset J$ and thus $t_\alpha \in J$, and hence $t_v = t_\alpha^* t_\alpha \in J$. So for any $\lambda \in \wt\Lambda$, we have $t_\lambda = t_{s(\lambda)}t_\lambda \in J$. So $\{t_\lambda : \lambda \in \wt\Lambda\} \subset J$, and thus $J = C^*(\wt\Lambda)$.
\end{proof}

\section{The Diagonal and the Spectrum}\label{hrg diag}
For $k$-graph $\Lambda$, we call $C^*\{s_\mu s_\mu^* : \mu \in \Lambda\} \subset C^*(\Lambda)$ the \emph{diagonal} $C^*$-algebra of $\Lambda$ and denote it $D_\Lambda$, dropping the subscript when confusion is unlikely. For a commutative $C^*$-algebra $A$, denote by $\Delta(A)$ the spectrum of $A$. Given a homomorphism $\pi:A\to B$ of commutative $C^*$-algebras, define by $\pi^*$ the induced map from $\Delta(B)$ to $\Delta(A)$ such that $\pi^*(f)(y) = f(\pi(y))$ for all $f \in \Delta(B)$ and $y \in A$.

\begin{theorem}\label{mainthm1hrgraph}
 Let $\Lambda$ be a row-finite higher-rank graph. Let $p\in M(C^*(\wt\Lambda))$ and $\varsigma:C^*(\Lambda) \cong pC^*(\wt\Lambda)p$ be from Theorem~\ref{c*lambda morita to c*wtlambda}. Then the restriction $\varsigma|_{D_\Lambda} =:\rho$ is an isomorphism of $D_\Lambda$ onto $pD_{\wt\Lambda}p$. Let $\pi: \iota(\Lambda^0) \wt\Lambda^\infty \to \iota(\partial \Lambda)$ be the homeomorphism from Theorem~\ref{hrg path spaces homeo}, then there exist homeomorphisms $h_\Lambda:\partial \Lambda\to \Delta(D_\Lambda)$ and $\eta:\iota(\Lambda^0) \wt\Lambda^\infty \to \Delta(p D_{\wt\Lambda} p)$ such that the following diagram commutes.

\begin{center}
\begin{tikzpicture}[>=stealth,scale=0.7]
    \node (e0f) at (0,2) {$\iota(\Lambda^0) \wt\Lambda^\infty$};
    \node (deltapdp) at (0,0) {$\Delta(pD_{\wt\Lambda}p)$}
        edge[<-] node[auto] {$\eta$} (e0f);
    \node (partiale) at (3,2) {$\iota(\partial \Lambda)$}
        edge[<-] node[auto,swap] {$\pi$} (e0f);
    \node (deltad) at (3,0) {$\Delta(D_\Lambda)$}
        edge[<-] node[auto,swap] {$h_\Lambda \circ \iota^{-1}$} (partiale)
        edge[<-] node[auto] {$\rho^*$} (deltapdp);
\end{tikzpicture}
\end{center}
\end{theorem}

As in \cite{RS2005}, for a finite subset $F \subset \Lambda$, define \[\vee F := \bigcup_{G \subset F} \MCE(G) = \bigcup_{G \subset F}\big\{\lambda \in \bigcap_{\mu\in G} \mu\Lambda : d(\lambda) = \bigvee_{\mu \in G}d(\mu)\big\}.\]

\begin{lemma}\label{hrg_L_snusnu*=sumqnu}
Let $\Lambda$ be a finitely aligned k-graph and let $F$ be a finite subset of $\Lambda$. Suppose that $r(\lambda) \in F$ for each $\lambda \in F$. For $\mu \in F$, define
\[
    q_\mu^{\vee F}:= s_\mu s_\mu^* \prod_{\mu \mu' \in \vee F \setminus \{\mu\}}(s_\mu s_\mu^* - s_{\mu\mu'} s_{\mu\mu'}^*).
\]
Then the $q_\mu^{\vee F}$ are mutually orthogonal projections in $\lsp \{s_\mu s_\mu^* : \mu \in \vee F\}$, and for each $\nu \in \vee F$
\begin{equation}\label{hrg eq_snusnu*=sumqnu}
s_\nu s_\nu^* = \sum_{\nu\nu' \in \vee F} q_{\nu\nu'}^{\vee F}
\end{equation}
\begin{proof}
Since
\[
s_\mu s_\mu^* \prod_{\mu \mu' \in \vee F \setminus \{\mu\}}(s_\mu s_\mu^* - s_{\mu\mu'} s_{\mu\mu'}^*) = s_\mu s_\mu^* \prod_{\mu \mu' \in \vee F, d(\mu') \neq 0}(s_{r(\mu)} - s_{\mu\mu'} s_{\mu\mu'}^*) ,
\]
\cite[Proposition 8.6]{RS2005} says precisely that the $q_\mu^{\vee F}$ are mutually orthogonal projections. That
\[
s_\nu s_\nu^* = \sum_{\nu\nu' \in \vee F} q_{\nu\nu'}^{\vee F}
\]
is established in the proof of \cite[Proposition 8.6]{RS2005} on page $421$.
\end{proof}
\end{lemma}

\begin{rmk}
We have
\[
    q_\mu^{\vee F} = s_\mu \Big( \prod_{\substack{\mu' \in s(\mu)\Lambda\setminus\{s(\mu)\}\\\mu\mu' \in {\vee F}}}(s_{s(\mu)} - s_{\mu'} s_{\mu'}^*)\Big)s_\mu^*.
\]
This follows from a straightforward induction on $|\vee F|$.
\end{rmk}

The following lemma can be verified through routine calculation. The reader is referred to the author's PhD thesis for details.
\begin{lemma}[{\cite[Lemma~A.0.7]{WebsterPhD}}]\label{lemma about projections nonzero}
Let $A$ be a $C^*$-algebra, let $p$ be a projection in $A$, let $Q$ be a finite set of commuting subprojections of $p$ and let $q_0$ be a nonzero subprojection of $p$. Then $\prod_{q\in Q}(p-q)$ is a projection. If $q_0$ is orthogonal to each $q \in Q$, then $q_0\prod_{q\in Q}(p-q) = q_0$, so in particular, $\prod_{q\in Q}(p-q) \neq 0$.
\end{lemma}

\begin{prop}\label{hrg_boundary homeo to spectrum of diagonal}
Let $\Lambda$ be a finitely aligned $k$-graph. Then $D = \clsp\{s_\mu s_\mu^* : \mu \in \Lambda\}$, and for each $x \in \partial \Lambda$ there exists a unique $h(x) \in \Delta(D)$ such that \[
h(x)(s_\mu s_\mu^*) =
    \begin{cases}
        1 &\text{if $x = \mu\mu'$}\\
        0 &\text{otherwise}.
    \end{cases}
\]
Moreover, $x \mapsto h(x)$ is a homeomorphism $h:\partial \Lambda \to \Delta(D)$.
\begin{proof}
Let $\mu,\nu \in \Lambda$. It follows from \eqref{ck3} that
\[
    (s_\mu s_\mu^*)(s_\nu s_\nu^*) = \sum_{\lambda \in \MCE(\mu,\nu)} s_\lambda s_\lambda^*,
\]
hence $D = \clsp \{s_\mu s_\mu^* : \mu \in \Lambda\}$.

Fix $x\in\partial \Lambda$ and $\sum_{\mu \in F}b_\mu s_\mu s_\mu^* \in \lsp\{s_\mu s_\mu^* : \mu \in \Lambda\}$. By setting extra coefficients to zero we can assume that each path in $F$ has its range in $F$, and write
\[
    \sum_{\mu \in F}b_\mu s_\mu s_\mu^* = \sum_{\mu \in \vee F}b_\mu s_\mu s_\mu^*.
\]
Let $n = \bigvee\{p \in \NN^k : x(0,p) \in \vee F\}$. Since $\vee F$ is a finite set of finite paths, $n$ is finite. Since $\vee F$ is closed under minimal common extensions, $x(0,n) \in \vee F$. Furthermore, since $x \in \partial\Lambda$, we have
\[
F_x := \{\mu' \in x(n)\Lambda\setminus\{x(n)\} : x(0,n)\mu' \in \vee F\} \notin x(n)\FE(\Lambda).
 \]
So there exists $\nu \in x(n)\Lambda$ such that for each $\mu' \in F_x$, $\MCE(\nu,\mu') = \emptyset$. Then $s_\nu s_\nu^* s_{\mu'} s_{\mu'}^* = 0$ for all $\mu' \in F_x$. Applying Lemma~\ref{lemma about projections nonzero} with $p = s_{x(n)}$, $q_0 = s_\nu s_\nu^*$ and $Q = \{s_{\mu'}s_{\mu'}^* : \mu' \in F_x\}$, we have $\prod_{\mu'\in F_x}(s_{x(n)} - s_{\mu'}s_{\mu'}^*) \neq 0$. So
\[
q_{x(0,n)}^F = s_{x(0,n)} \prod_{\mu'\in F_x}(s_{x(n)} - s_{\mu'}s_{\mu'}^*) s_{x(0,n)}^*  \neq 0.
 \]
We have
\begin{align*}
\Big\| \sum_{\nu \in \vee F} b_\mu s_\mu s_\mu^* \Big\| &= \Big\| \sum_{\nu \in \vee F} \Big(\sum_{\substack{\mu \in \vee F\\\nu \in \Zz(\mu)}} b_\mu\Big) q^{\vee F}_\nu \Big\| \qquad\text{by~\eqref{hrg eq_snusnu*=sumqnu}}\\
&= \max_{\{\nu \in \vee F: q^{\vee F}_\nu \neq 0\}} \Big|\sum_{\substack{\mu \in \vee F\\\nu \in \Zz(\mu)}} b_\mu\Big|\\
&\geq\Big|\sum_{\substack{\mu \in \vee F\\x(0,n) \in \Zz(\mu)}} b_\mu\Big|\qquad\text{since $q_{x(0,n)}^{\vee F} \neq 0$}\\
&= \Big|\sum_{\substack{\mu \in F\\x(0,n) \in \Zz(\mu)}} b_\mu\Big| \qquad\text{since $b_\mu = 0$ for $\mu \in \vee F\setminus F$}.
\end{align*}
Hence the formula
\begin{equation}\label{hrg_want h(x) to be this}
h(x)\Big(\sum_{\mu \in F}b_\mu s_\mu s_\mu^*\Big) = \sum_{\substack{\mu\in F\\x \in \Zz(\mu)}} b_\mu,
\end{equation}
determines a norm-decreasing linear map on $\lsp\{s_\mu s_\mu^* : \mu \in \Lambda\}$.

To see that $h(x)$ is a homomorphism, it suffices to show that
\begin{equation}\label{hrg_h is multiplicative}
 h(x)(s_\mu s_\mu^* s_\alpha s_\alpha^*) = h(x)(s_\mu s_\mu^*) ~h(x)( s_\alpha s_\alpha^*).
\end{equation}
Calculating the right hand side of \eqref{hrg_h is multiplicative} yields
\[
h(x)(s_\mu s_\mu^*) ~h(x)(s_\alpha s_\alpha^*) =
\begin{cases}
  1 &\text{if $x \in \Zz(\mu) \cap \Zz(\alpha)$}\\
  0 &\text{otherwise.}
\end{cases}
\]

Calculating the left hand side of \eqref{hrg_h is multiplicative} gives
\[
    h(x)(s_\mu s_\mu^* s_\alpha s_\alpha^*) = h(x)\Big(\sum_{\lambda \in \MCE(\mu,\alpha)} s_\lambda s_\lambda^*\Big).
\]
There exists at most one $\lambda \in \MCE(\mu,\alpha)$ such that $x \in \Zz(\lambda)$. Such a $\lambda$ exists if and only if $x \in \Zz(\mu) \cap \Zz(\alpha)$, so
\[
    h(x)(s_\mu s_\mu^* s_\alpha s_\alpha^*)=\begin{cases}1 &\text{if $x \in \Zz(\alpha)\cap\Zz(\mu)$}\\ 0 &\text{otherwise.}\end{cases}
\]
Thus we have established \eqref{hrg_h is multiplicative}, hence $h(x)$ is a homomorphism, and thus extends uniquely to a nonzero homomorphism $h(x):D \to \CC$.

We claim the map $h:\partial \Lambda \to \Delta(D)$ is a homeomorphism. The trickiest part is to show $h$ is onto:
\begin{claim}
The map $h$ is surjective.
\begin{proof}
Fix $\phi \in \Delta(D)$. We seek $x \in \partial \Lambda$ such that $h(x) = \phi$. For each $n \in \NN^k$, $\{s_\mu s_\mu^* : d(\mu) =n \}$ are mutually orthogonal projections. It follows that for each $n \in \NN^k$ there exists at most one $\nu^n \in \Lambda^n$ such that $\phi(s_{\nu^n} s_{\nu^n}^*)=1$. Let $S$ denote the set of $n$ for which such $\nu^n$ exist. If $\nu = \mu\nu'$ and $\phi(s_\nu s_\nu^*)=1$, then
\[
    1= \phi(s_\nu s_\nu^*) = \phi(s_\nu s_\nu^* s_\mu s_\mu^*) = \phi(s_\nu s_\nu^*) \phi(s_\mu s_\mu^*) = \phi(s_\mu s_\mu^*).
\]
This implies that if $n \in S$ and $m \leq n$, then $m \in S$ and $\nu^m = \nu^n(0,m)$. Set $N:= \vee S$, and define $x:\Omega_{k,N} \to \Lambda$ by $x(p,q) = \nu^q(p,q)$. Then since each $\nu^q$ is a $k$-graph morphism, so is $x$.

We now show that $x \in \partial\Lambda$. Fix $n \in \NN^k$ such that $n\leq d(x)$, and $E \in x(n)\FE(\Lambda)$. We seek $m\in \NN^k$ such that $x(n,n+m) \in E$. Since $E$ is finite exhaustive, \eqref{ck4} implies that
$
    \prod_{\lambda \in E}(s_{x(n)} - s_\lambda s_\lambda^*) = 0.
$
Multiplying on the left by $s_{x(0,n)}$ and on the right by $s_{x(0,n)}^*$ yields
\[
    \prod_{\lambda \in E}(s_{x(0,n)}s_{x(0,n)}^* - s_{x(0,n)\lambda} s_{x(0,n)\lambda}^*) = 0.
\]
Thus, since $\phi$ is a homomorphism, there exists $\lambda \in E$ such that
\begin{align*}
0 &= \phi(s_{x(0,n)}s_{x(0,n)}^*) - \phi(s_{x(0,n)\lambda} s_{x(0,n)\lambda}^*)\\
&= \phi(s_{\nu^n} s_{\nu^n}^*) - \phi(s_{x(0,n)\lambda} s_{x(0,n)\lambda}^*)\\
&= 1 - \phi(s_{x(0,n)\lambda} s_{x(0,n)\lambda}^*)
\end{align*}
So $\phi(s_{x(0,n)\lambda} s_{x(0,n)\lambda}^*) = 1$. Thus $x(0,n)\lambda = \nu^{n+d(\lambda)} = x(0,n+d(\lambda))$, and hence $x \in \partial\Lambda$.

Now we must show that $h(x) = \phi$. For each $\mu \in \Lambda$ we have
\begin{align*}
    \phi(s_\mu s_\mu^*) = 1 &\iff d(\mu) \in S \text{ and }\nu^{d(\mu)} = \mu \\&\iff x(0,d(\mu))=\mu \\&\iff h(x)(s_\mu s_\mu^*) = 1.
\end{align*}
Since $\phi(s_\mu s_\mu^*)$ and $h(x)(s_\mu s_\mu^*)$ take values in $\{0,1\}$, we have $h(x) = \phi$.
\pc\end{proof}
\end{claim}

To see that $h$ is injective, suppose that $h(x) = h(y)$. Then for each $n \in \NN^k$, we have
\[
h(y)(s_{x(0,n \wedge d(x))} s_{x(0,n \wedge d(x))}^*) = h(x)(s_{x(0,n \wedge d(x))} s_{x(0,n \wedge d(x))}^*) = 1.\]%\text{ and}
%h(x)(s_{y(0,n \wedge d(y))} s_{y(0,n \wedge d(y))}^*) &= h(y)(s_{y(0,n \wedge d(y))} s_{y(0,n \wedge d(y))}^*) = 1.
%\end{align*}
Hence $y(0,n \wedge d(x)) = x(0,n \wedge d(x))$. By symmetry, we also have $y(0,n \wedge d(y)) = x(0,n \wedge d(y))$ for all $n$. In particular, $d(x) = d(y)$ and $y(0,n) = x(0,n)$ for all $n \leq d(x)$. Thus $x=y$.

Recall that $\Delta(D)$ carries the topology of pointwise convergence. For openness, it suffices to check that $h^{-1}$ is continuous. Suppose that $h(x^n) \to h(x)$. Fix a basic open set $\Zz(\mu)$ containing $x$, so $h(x)(s_\mu s_\mu^*) = 1$. Since $h(x^n)(s_\mu s_\mu^*) \in \{0,1\}$ for all $n$, for large enough $n$, we have $h(x^n)(s_\mu s_\mu^*) = 1$. So $x^n \in \Zz(\mu)$. For continuity, a similarly straightforward argument shows that if $x^n \to x$, then $h(x^n)(s_\mu s_\mu^*) \to h(x)(s_\mu s_\mu^*)$. This convergence extends to $\lsp\{s_\mu s_\mu^* : \mu \in \Lambda\}$ by linearity, and to $D$ by an $\eps/3$ argument.

\end{proof}
\end{prop}

We can now prove our main result.

\begin{proof}[Proof of Theorem~\ref{mainthm1hrgraph}.]

Let $\Lambda$ be a row-finite $k$-graph, and $\wt\Lambda$ be the desourcification described in Proposition~\ref{facprop}. Let $\{s_\lambda:\lambda \in \Lambda\}$ and $\{t_\lambda: \lambda \in \wt\Lambda\}$ be universal Cuntz-Krieger families in $C^*(\Lambda)$ and $C^*(\wt\Lambda)$. Let $A$ be the $C^*$-subalgebra of $C^*(\wt\Lambda)$ generated by $\{t_\lambda : \lambda \in \iota(\Lambda)\}$, and define the diagonal subalgebra of $A$ by $D_A:= \clsp\{t_\lambda t_\lambda^* : \lambda \in \iota(\Lambda)\}$. Replacing $t_\lambda t_\mu^*$ with $t_\lambda t_\lambda^*$ in the proof Theorem~\ref{c*lambda morita to c*wtlambda} yields $D_A \cong p D_{\wt\Lambda} p$. Since $A \cong C^*(\Lambda)$, it follows that $D_A \cong D_\Lambda$. Thus $D_\Lambda \cong p D_{\wt\Lambda} p$ as required.

We now construct $\eta$ and show that it is a homeomorphism. That $p$ commutes with $D_{\wt\Lambda}$ implies that $p D_{\wt \Lambda} p$ is an ideal in $D_{\wt\Lambda}$. Then \cite[Propositions A26(a) and A27(b)]{RW1998} imply that map $k:\phi \mapsto \phi|_{p D_{\wt \Lambda} p}$ is a homeomorphism of $\{\phi\in\Delta(D_{\wt \Lambda}): \phi|_{p D_{\wt \Lambda} p} \neq 0\}$ onto $\Delta(p D_{\wt \Lambda} p)$.  Since $\wt\Lambda$ is row finite with no sources, $\partial \wt\Lambda = \wt\Lambda^\infty$. Let $h_{\wt\Lambda}:\wt\Lambda^\infty \to \Delta(D_{\wt\Lambda})$ be the homeomorphism obtained from Proposition~\ref{hrg_boundary homeo to spectrum of diagonal}. Then $h_{\wt\Lambda}(x) \in \dom(k)$ for all $x \in \iota(\Lambda^0) \wt\Lambda^\infty$. Define $\eta:=k \circ h_{\wt\Lambda}|_{\iota(\Lambda^0)\wt\Lambda^\infty}: \iota(\Lambda^0) \wt\Lambda^\infty \to \Delta(p D_{\wt\Lambda} p)$.

We now show that $h_\Lambda \circ \iota^{-1} \circ \pi= \rho^*\circ \eta.$ Since $\rho$ is an isomorphism, it suffices to fix $x \in \iota(\Lambda^0)\wt\Lambda^\infty$ and $\mu \in \Lambda$ and show that

\begin{equation}\label{hrg intertwining diagram commutes}
(h_\Lambda \circ \iota^{-1} \circ \pi) (x) (s_\mu s_\mu^*) = (\rho^*\circ \eta)(x) (s_\mu s_\mu^*).
\end{equation}
Let $\omega \in \partial\Lambda$ be such that $\pi(x) = \iota(\omega)$. Then the left-hand side of \eqref{hrg intertwining diagram commutes} becomes
\[
(h_\Lambda \circ \iota^{-1} \circ \pi) (x) (s_\mu s_\mu^*) = h_\Lambda(w)(s_\mu s_\mu^*) =
\begin{cases}
    1 &\text{if } \omega \in \Zz(\mu)\\
    0 &\text{otherwise.}
\end{cases}
\]
Since $r(x) \in \iota(\Lambda^0)$, the right-hand side of \eqref{hrg intertwining diagram commutes} simplifies to
\[
(\rho^*\circ \eta)(x) (s_\mu s_\mu^*) = \eta(x)(\rho(s_\mu s_\mu^*)) = h_{\wt\Lambda}(x)(t_{\iota(\mu)} t_{\iota(\mu)}^*) = \begin{cases}
    1 &\text{if } x \in \Zz(\iota(\mu))\\
    0 &\text{otherwise.}
\end{cases}
\]

We claim that $x \in \Zz(\iota(\mu))$ if and only if $\omega \in \Zz(\mu)$. Suppose that $x \in \Zz(\iota(\mu))$. Since $\mu \in \Lambda$ and $\pi(x) = \iota(\omega)$, we have $\pi(x(0,d(\mu))) = \pi(\iota(\mu)) = \iota(\mu)$. So $d(\pi(x(0,d(\mu)))) = d(\mu)$, and thus $d(x) \wedge d(w) \geq d(\mu)$. So $d(\omega) \geq d(\mu)$. Then we have
\begin{align*}
x \in \Zz(\iota(\mu)) &\iff x(0,d(\mu)) = \iota(\mu) \qquad\text{since $\iota$ preserves degree}\\
&\iff [\omega;(0,d(\mu))] = \iota(\mu) \qquad\text{by Lemma~\ref{key to pi inj}}\\
&\iff \iota(\omega(0,d(\mu))) = \iota(\mu) \qquad\text{by Remark~\ref{iota of a segment}}\\
&\iff \omega(0,d(\mu)) = \mu \qquad\text{since $\iota$ is a injective}\\
&\iff \omega \in \Zz(\mu).
\end{align*}
So equation \eqref{hrg intertwining diagram commutes} holds, and we are done.
\end{proof}

\end{document}